\documentclass[11pt]{amsart}

\usepackage{amssymb} 
\usepackage[mathscr]{eucal} 
\usepackage[matrix,arrow,tips,curve,ps]{xy}
\usepackage{amsmath} 
\usepackage{epsfig}
\usepackage{amscd}
\usepackage{tikz}
\usepackage{tikz-cd}
\usepackage{mathrsfs}
\usepackage{verbatim}

\newtheorem{TEO}{Theorem}[section]

\newtheorem{LEM}[TEO]{Lemma}
\newtheorem{DEF}[TEO]{Definition}

\newtheorem{REM}[TEO]{Remark}

\newtheorem*{que}{Question}

\theoremstyle{definition}
\newtheorem{remark}[TEO]{Remark}
\newtheoremstyle{dico}
 {\baselineskip}   
  {\topsep}   
  {}  
  {0pt}       
  {} 
  {.}         
  {5pt plus 1pt minus 1pt} 
  {}          
\theoremstyle{dico}
\newtheorem{say}[TEO]{}
\numberwithin{equation}{section}

\def\OO{{\mathcal O}}

\newcommand\Ga{\Gamma}

\newcommand\dual{\mathrel{\raise3pt\hbox{$\underline{\mathrm{\thinspace d
\thinspace}}$}}}

\newcommand\proj{\mathbb P}
\newcommand\Z{\mathbb Z}
\newcommand\ZZ{\mathbb Z}

\newcommand\R{\mathbb R} \newcommand\Q{\mathbb Q}
\newcommand\Co{\mathbb C}

\newcommand\A{{\mathsf A}}
\newcommand\T{{\mathsf T}}

\newcommand\Pic{\operatorname{Pic}}

\newcommand{\RR}{\mathsf{R}}
\newcommand{\rrd}{\RR_{g,[2]}}
\newcommand{\RRd}{\RR_{[2]}}
\newcommand{\Rg}{\mathsf{R}_g}
\newcommand{\Nm}{\operatorname{Nm}}
\newcommand{\Tei}{{\mathsf{T}}}

\newcommand{\M}{{\mathsf{M}}}

\newcommand{\meno}{^{-1}}
\newcommand{\Aut}{\operatorname{Aut}}
\newcommand{\End}{\operatorname{End}}

\newcommand{\restr}[1]          {\vert_{#1}}

\newcommand{\PP}{{\mathbb P}}

\renewcommand{\phi}{\varphi}
\newcommand{\lds}{\ldots}
\newcommand{\cds}{\cdots}
\newcommand{\cd}{\cdot}

\newcommand{\sx}{\langle}
\newcommand{\xs}{\rangle}
\newcommand{\lra}{\longrightarrow}
\newcommand{\ra}{\rightarrow}
\newcommand{\ga}{\gamma}
\newcommand{\alfa}{\alpha}

\newcommand{\vacuo}{\emptyset}

\newcommand{\La}{\Lambda}

\newcommand{\GL}{\operatorname{GL}}
\newcommand{\HH}{\mathfrak{H}}
\newcommand{\sieg}{\HH_g}

\newcommand{\Sp}                {\operatorname {Sp}}

 \newcommand{\mm}{{\mathbf{m}}}

 \newcommand{\ag}{\mathsf{A}_g}
 
 \newcommand{\zg}{\mathsf{Z}}
 \newcommand{\datum}{{(\mm, G, \theta)}}
\newcommand{\datu}{{( G, \theta)}}
\newcommand{\tdatu}{{( \tG, \ttheta)}}
\newcommand{\x}{\Xi}
\newcommand{\y}{\Xi}
\newcommand{\dat}{{( \tilde{G}, \ttheta, \sigma)}}

 \newcommand{\braid}{\mathbf{B_r}}

\newcommand{\TG}{{\tilde{G}}}
\newcommand{\tG}{{\tilde{G}}}
\newcommand{\tg}{{\tilde{g}}}
\newcommand{\ttheta}{{\tilde{\theta}}}
\newcommand{\Mod}{\operatorname{Mod}}
\newcommand{\gtheta}{{G_\theta}}

\newcommand{\pr}{\mathscr{P}}  
\newcommand{\tj}{\widetilde{j}}  
\newcommand{\tC} {{\tilde{C} }}

\begin{document}

 \title{Shimura curves in the Prym locus}

 \author[E. Colombo]{Elisabetta Colombo} \address{Dipartimento di
   Matematica, Universit\`a di Milano, via Saldini 50, I-20133,
   Milano, Italy } \email{{\tt elisabetta.colombo@unimi.it}}

 \author[P. Frediani]{Paola Frediani} \address{ Dipartimento di
   Matematica, Universit\`a di Pavia, via Ferrata 5, I-27100 Pavia,
   Italy } \email{{\tt paola.frediani@unipv.it}}

 \author[A. Ghigi]{Alessandro Ghigi} \address{Dipartimento di
   Matematica, Universit\`a di Pavia, via Ferrata 5, I-27100, Pavia,
   Italy } \email{{\tt alessandro.ghigi@unipv.it}}

 \author[M. Penegini]{Matteo Penegini} \address{Dipartimento di
   Matematica, Universit\`a di Genova, Via Dodecaneso 35, I-16146
   Genova, Italy } \email{{\tt penegini@dima.unige.it}}

 \begin{abstract}
   We study Shimura curves of PEL type in $\mathsf{A}_g$ generically
   contained in the Prym locus.  We study both the unramified Prym
   locus, obtained using \'etale double covers, and the ramified Prym
   locus, corresponding to double covers ramified at two points. In
   both cases we consider the family of all double
   covers 
   compatible with a fixed group action on the base curve.  We
   restrict to the case where the family is 1-dimensional and the
   quotient of the base curve by the group is $\proj^1$.  We give a
   simple criterion for the image of these families under the Prym map
   to be a Shimura curve.  Using computer algebra we check all the examples gotten in this way up to genus 28.  We obtain 43 Shimura
   curves  contained in the unramified Prym locus and
   9 families contained in the ramified Prym
   locus.  Most of these curves are not generically contained in the
   Jacobian locus. 

 \end{abstract}

 \thanks{The first and the fourth authors were partially supported by MIUR PRIN 2015``Geometry of Algebraic Varieties''.
   The second and third authors were partially supported by MIUR PRIN
   2015 ``Moduli spaces and Lie theory''.  
  The second 
   author was also partially supported by FIRB 2012 `` Moduli Spaces and their Applications''. The third
   author was also supported by FIRB 2012 ``Geometria differenziale e
   teoria geometrica delle funzioni''.  The authors were also
   partially supported by GNSAGA of INdAM.  }

 \maketitle

 \section{Introduction}

 Denote by $ \Rg $ the scheme of isomorphism classes $ [C, \eta]$,
 where $ C$ is a smooth projective curve of genus $g$ and
 $\eta \in \Pic^0(C)$ is such that $ \eta^2 = \OO_C$ and
 $ \eta \neq \OO_C $.  A point $[C,\eta] $ corresponds to an \'etale
 double cover $ h:\tilde{C} \lra C$.  The norm map
 $\Nm : \Pic^0(\tilde{C}) \lra \Pic^0(C)$ is defined by
 $ \Nm( \sum_i a_i p_i) = \sum_i a_i h(p_i)$.  The Prym variety
 associated to $[C,\eta]$ is the connected component containing $0$ of
 $\ker \Nm$. It is a principally polarized abelian variety of
 dimension $g-1$, denoted by $P(C,\eta)$ or equivalently
 $P(\tilde{C}, C)$.  This defines the Prym map
 \begin{gather*}
   \pr : \Rg \lra \mathsf{A}_{g-1}, \qquad \pr ([C,\eta] ) := [ P(C,
   \eta)].
 \end{gather*}

 where $\mathsf{A}_{g-1}$ is the moduli space of principally polarized
 abelian varieties of dimension $g-1$. We recall that the Prym map is
 generically an embedding for $g \geq 7$ \cite{friedman-smith},
 \cite{kanev-global-Torelli} and it is generically finite for
 $g \geq 6$.  The Prym map is never injective and it has positive
 dimensional fibres \cite{donagi}, \cite{mumford}, \cite{carlos}.

 Analogously one can consider the moduli space parametrising ramified
 double coverings and the corresponding Prym varieties. We will only
 consider the case in which the Prym variety is principally polarised,
 that is when the map is ramified at two distinct points.

 So let $\RR_{g,[2]}$ denote the scheme parametrizing triples
 $[C, \eta, B]$ up to isomorphism, where $C$ is a genus $g$ curve,
 $\eta$ a line bundle on $C$ of degree 1, and $B$ a reduced divisor in
 the linear system $|{\eta^2}|$ corresponding to a $2:1$ covering
 $\pi:\tC \rightarrow C$ ramified over $B$. The \emph{Prym map} is the
 morphism
 \[ {\pr}: \RR_{g,[2]} \ra \mathsf{A}_{g}
 \]
 which associates to $[C, \eta, B]$ the Prym variety $P(\tC,C)$ of
 $\pi$. It is generically finite for $g \geq 5$ and generically
 injective for $g \geq 6$ (see \cite{pietro-vale}).

 Denote by
 \begin{gather*}
   j: \mathsf{M}_g \lra \mathsf{A}_{g}, \qquad j([C] ) := [ J(C)].
 \end{gather*}
 the Torelli map and by $\overline{j(\mathsf{M}_g )}$ the Torelli
 locus. The work of Beauville \cite{b} on admissible covers shows that
 one has the following inclusions
 \begin{gather}
   \label{inclu}
   \overline{j(\mathsf{M}_g )} \subset \overline{ \pr (\RR_{g,[2]}) }
   \subset \overline{ \pr (\RR_{g+1}) }.
 \end{gather}
 (See also \cite{fl} and in the ramified case \cite{pietro-vale} and
 also sections \ref{prymetale} and \ref{sec:ramified} below).

 On $\ag$, viewed as orbifold, there is a natural variation of Hodge
 structure whose fiber at a point $A$ is $H^1(A, \Q)$. The Hodge loci
 for this variation of Hodge structure are called \emph{special} or
 \emph{Shimura subvarieties} of $\ag$.  A conjecture by Coleman and
 Oort \cite{oort-can} says that for large genus there should not exist
 special or Shimura subvarieties of $\ag$ \emph{generically contained}
 in the Torelli locus, i.e. contained in $\overline{j(\mathsf{M}_g )}$
 and intersecting $j(\mathsf{M}_g )$.  See \cite{moonen-oort} for more
 information, \cite{hain,dejong-zhang,cfg,liu-yau,lz} for some results
 towards the conjecture and
 \cite{dejong-noot,moonen-special,fgp,fpp,gm1,gm2} for counterexamples
 to the conjecture in low genera.

 Recall that Shimura subvarieties of $\ag$ are totally geodesic with
 respect to the orbifold metric induced on $\ag$ from the symmetric
 metric on the Siegel space $\sieg$. The conjecture is coherent with
 the fact that the Torelli locus is very curved, and a possible
 approach to the conjecture is via the study of the second fundamental
 form of the Torelli map (\cite{cpt},\cite{cfg}).  The geometry of
 $\Rg$ has many analogies with the geometry of $\mathsf{M}_g$ and it
 has been extensively investigated (see \cite{Farkas} for a nice
 survey). Moreover, the second fundamental form of the Prym map
 $ \pr : \Rg \lra \mathsf{A}_{g-1}$ has a very similar structure and
 similar properties as the one of the Torelli map \cite{cfprym}.

 In view of these similarities and of the inclusions \eqref{inclu} it
 is natural to ask the question below, which is analogous to the one
 of Coleman and Oort, for the Prym loci $\overline{\pr (\RR_{g+1})}$
 and $\overline{\pr (\RR_{g,[2]})}$. We say that a subvariety
 $Z\subset \A_g$ is \emph{generically contained} in the Prym locus
 $\pr (\RR_{g+1})$ if $Z \subset \overline{\pr (\RR_{g+1})}$,
 $Z \cap \pr (\RR_{g+1})\neq \vacuo$ and $Z $ intersects the locus of
 irreducible principally polarized abelian varieties.  The same
 terminology applies for $\pr (\RR_{g,[2]})$.

 \begin{que}
   Do there exist special subvarieties of $\A_g$ that are generically
   contained in the Prym loci $\pr (\RR_{g+1})$ and
   $\pr (\RR_{g,[2]})$ for $g$ sufficiently high?
 \end{que}

 As in the case of the Torelli locus, the condition of being
 generically contained in the Prym locus ensures that examples of
 Shimura varieties in a given dimension are not inductively
 constructed from Shimura varieties in lower dimension. Notice in fact
 that in the Prym case it is possible to construct Prym varieties
 obtained by \'etale covers of smooth hyperelliptic curves which are
 reducible as principally polarized abelian varieties, see
 \cite[p. 344]{mumford}.

 For low genera ($g \leq 7$) there do exist Shimura subvarieties of
 $\ag$ contained in the Torelli locus. These have all been constructed
 as families of Jacobians of Galois coverings of $\PP^1$ and of genus
 one curves (\cite{dejong-noot}, \cite{rohde}, \cite{moonen-special},
 \cite{moonen-oort}, \cite{fgp}, \cite{fpp}) \cite{gm1},
 \cite{gm2}). All these families of curves $C$ satisfy the sufficient
 condition that $\dim(S^2 H^0(K_C))^G = \dim H^0(2K_C)^G$, where $G$
 is the Galois group of the covering (see \cite{fgp} Theorem 3.9).
 This condition ensures that the multiplication map
 $m: (S^2 H^0(K_C))^G \rightarrow H^0(2K_C)^G$ is an
 isomorphism. Notice that the multiplication map is the codifferential
 of the Torelli map.  As a first attempt to see the similarity between
 the Torelli and Prym loci from this point of view, in this paper we
 construct Shimura curves contained in the Prym loci that satisfy an
 analogous sufficient condition.

The following statement summarises our results:   
\begin{TEO}
  In the unramified case there are 43 families of Pryms yielding
  Shimura curves of $\A_{g-1}$ for $g \leq 13$.  The generic Prym in
  all the families for $g \geq 4$ is irreducible.

  In the ramified case there are  8 families of Pryms yielding
  Shimura curves of $\A_{g}$ with $g \leq 8$. \end{TEO}

We now describe our construction  in detail. We consider a one-dimensional family of curves
 $\{\tC_t\}_{t\in \Co - \{0,1\}}$ admitting an action of a group of
 automorphisms $\tG$ containing a central involution $\sigma$ and such
 that the quotient $\tC_t/\tG \cong \PP^1$, the covering
 $\psi_t: \tC_t \rightarrow \tC_t/\tG$ is branched at 4 points and the
 double covering $\tC_t \rightarrow \tC_t/\langle \sigma \rangle=:C_t$
 is either \'etale or ramified at two distinct points. We give a
 condition which ensures that the family of the Prym varieties
 $P(\tC_t, C_t)$ of the 2:1 coverings yields a Shimura curve.  The
 condition is that the multiplication map
 $m: (S^2 H^0(K_{C_t} \otimes \eta))^{\tG} \rightarrow H^0(2K_{C_t}
 \otimes 2\eta)^{\tG}$
 is an isomorphism. The multiplication map is the codifferential of
 the Prym map.

 Since the covering $\psi_t$ is branched at 4 points,
 $\dim(H^0(2K_{C_t} \otimes 2\eta)^{\tG}) =1$, so our first
 requirement is that $\dim((S^2 H^0(K_{C_t} \otimes \eta))^{\tG}) = 1$
 (condition \eqref{condA} of section 3 and section 4).

 Unlike the Torelli map, the Prym map has positive dimensional fibers,
 therefore condition \eqref{condA} is not enough to ensure that
 multiplication map $m$ is an isomorphism, or equivalently that $m$ is
 not zero (condition (B) of section 3 and section 4).

 Morevover we have to check that the family of Pryms is not contained
 in the set of reducible abelian varieties. We do it in the unramified
 case in dimension $\geq 4$ using the criterion given in
 \cite[p. 344]{mumford}.

 We notice that if $(S^2 H^0(K_{C_t} \otimes \eta))^{\tG}$ is
 generated by a decomposable tensor (condition \eqref{condB1} of
 sections 3 and 4) the multiplication map cannot be zero, hence
 condition (B) is satisfied.

 This happens in particular when the group $\tG$ is abelian, hence in
 this case it is enough to verify condition \eqref{condA} to have a
 Shimura curve. When $\tG$ is not abelian we study the geometry of
 some of these families satisfying condition \eqref{condA} and we
 prove that the families of Pryms are not constant, hence condition
 (B) is satisfied.

 As in the Torelli case, all the examples we found up to now are in
 low dimension, namely in $\A_g$ with $g \leq 12$. All the examples
 where the group is abelian are in $\A_g$ with $g \leq 10$. In the
 ramified case they are all in dimension $g \leq 8$.  We also notice
 that the last example we find satisfying conditions \eqref{condA} and
 \eqref{condB1} are in dimension $g =10$. To prove that the remaining
 examples satisfying \eqref{condA} yield Shimura curves we need ad hoc
 arguments. On the whole, the number of examples satisfying condition
 \eqref{condA} decreases as the dimension grows. This suggests that,
 as in the Torelli case, one could expect that for high dimension
 there should not exist Shimura curves contained in the Prym locus
 constructed in this way.

 Let us explain explicitly how we construct these families in the case
 of unramified double coverings.

 A Galois covering $\tC \rightarrow \PP^1$ is determined by the Galois
 group $\tG$, an epimorphism $\ttheta : \Ga_r \ra \TG$ and the branch
 points $t_1,...,t_r \in \PP^1$ (see section 3 for the notation). We
 will choose $r=4$. We also fix a central involution $\sigma \in \tG$
 that does not lie in $\bigcup_{i=1}^r \sx \ttheta (\ga_i)\xs$.
 Denote by $G = \tG/\langle \sigma \rangle$. Fixing the Prym datum
 $(\tG, \ttheta, \sigma)$, setting $\{t_1,t_2,t_3\} = \{0,1, \infty\}$
 and letting the point $t_4 = t$ vary we get a one dimensional family
 of curves and coverings
 \begin{equation*}
   \begin{tikzcd}[row sep=tiny]
     \tilde{C}_t \arrow{rr}{\pi_ t} \arrow{rd} & &  \arrow{ld } C_t  = \tilde{C}_t /\sx \sigma\xs  \\
     & \PP^1\cong \tC_t/\tG \cong C_t/G &
   \end{tikzcd}
 \end{equation*}
 and correspondingly a family
 $\RR (\TG, \ttheta,\sigma) \subset \RR_g$.

 Let $\pi: \tC \rightarrow C$ be an element of the family and let
 $\eta \in Pic^0(C)$ be the 2-torsion element yielding the \'etale
 double covering $\pi$.  Set $V = H^0(\tilde{C}, K_{\tilde{C}})$, and
 let $V= V_+ \oplus V_-$ be the eigenspace decomposition for the
 action of $\sigma$.  The summand $V_+$ is isomorphic as a
 $G$-representation to $H^0(C, K_C)$, while $V_-$ is isomorphic to
 $H^0(C,K_C\otimes \eta)$.  Set $ W = H^0( \tC , 2K_\tC ) $ and let
 $W=W_+\oplus W_-$ be the eigenspace decomposition for the action of
 $\sigma$.  We have $W_+ \cong H^0(C,2K_C)$ and
 $W_- \cong H^0(C, 2K_C\otimes \eta)$.  Consider the multiplication
 map $ m : S^2 V \lra W.  $ It is the codifferential of the Torelli
 map $\tj : \M_\tg \ra \A_\tg$ at $[\tC] \in \M_\tg$.  The
 multiplication map is $\tG$-equivariant and we have the following
 isomorphisms
 \begin{gather*}
   (S^2V)^\tG = (S^2V_+)^G \oplus (S^2 V_-)^G, \ W^\tG = W_+ ^G.
 \end{gather*}
 Therefore $m$ maps $(S^2V)^\tG $ to $W_+^G$. We are interested in the
 restriction: \begin{gather}
   \label{condition}
   m : (S^2 V_-)^G \lra W_+^G.
 \end{gather}
 By the above discussion this is just the multiplication map
 $ (S^2 H^0 (C, K_C\otimes \eta))^G \lra H^0(C, 2K_C)^G.  $

 \begin{TEO} (see Theorem \ref{teo1})
   \label{TEOA}
   Let $\dat$ be a Prym datum. If the map $m$ in \eqref{condition} is
   an isomorphism, then the closure of $\pr ( \RR \dat )$ in
   $\A_{g-1}$ is a special subvariety contained in the Prym locus.
 \end{TEO}

 In a similar way one can construct families of Pryms in the ramified
 case and the analogous sufficient condition to ensure that the family
 yields a Shimura subvariety of $\ag$ (see Theorem \ref{teo1ram}).  To
 produce sistematically these Shimura families we used \verb|MAGMA|
 \cite{MA}. Our script is available at:
 \verb|http://www.dima.unige.it/~penegini/publ.html|.
 Using this script one can in principle determine all the families
 satisfying condition \eqref{condA} and \eqref{condB1} both in the
 unramified and in the ramified case for every $\tg = g(\tC)$.
  
 Notice that in the unramified case $\tg = 2g-1$, while in the
 ramified case $\tg = 2g$, where
 $g = g(C) = g(\tC/\langle \sigma \rangle)$.  As we have already
 observed, if $\tG$ is abelian, condition (B) is automatically
 satisfied, hence we get a Shimura curve. In the non abelian case we
 analysed some of the families satisfying condition \eqref{condA} and
 we proved that they also yield a Shimura curve.  
 
 The following  is a precise statement of our results.

\begin{TEO}
  In the unramified case, for $\tg = 2g-1 \leq 27$ we obtain 40
  families satisfying condition \eqref{condB1} (28 are abelian, 12
  non-abelian). We obtain three more non-abelian families satisfying
  condition \eqref{condB}, namely families 39, 42, 43 of Table 1. So
  in the unramified case we have found 43 families of Pryms yielding
  Shimura curves of $\A_{g-1}$ for $g \leq 13$.  The generic Prym in
  all the families for $g \geq 4$ is irreducible.

  In the ramified case, for $\tg = 2g \leq 28$, we found 9 Shimura
  families all with $\tg \leq 16$. Of these 9 families 6 satisfy
  condition \eqref{condB1}. Two other families do not satisfy
  condition \eqref{condB1}, but they satisfy condition \eqref{condB}.
  So in the ramified case we found 8 families of Pryms yielding
  Shimura curves of $\A_{g}$ with $g \leq 8$. See Table 1.
\end{TEO}

The plan of the paper is the following:

In section 2 we recall the definition of special or Shimura
subvarieties of $\ag$ and we briefly summarise some of the results of
section 3 of \cite{fgp}.

In section 3 we explain the construction of the families of Pryms in
the unramified case and we prove Theorem \ref{TEOA}.

In section 4 we do the analogous construction in the ramified case and
we prove the analogous result (Theorem \ref{teo1ram}).

Next we describe a sample of the examples.

All the unramified abelian examples are in $\A_k$ with $k\leq 10$. In
section 5 we describe the only 7 unramified abelian examples yielding
a Shimura curve generically contained in the Prym locus for
$k \geq 6$, hence for which the closure of the Prym locus is not all
$\A_k$.  There are two examples also for $k=8$, and one example for
$k =10$.  Up to now there are no known examples of Shimura varieties
generically contained in the Torelli locus in $\A_k$ for $k \geq 8$.
We also show that the familes in $\A_8$ are not families of Jacobians.
Next we describe three unramified non-abelian examples that don't
satisfy condition \eqref{condB1}. Hence we prove by ad hoc methods
that they do indeed produce Shimura curves generically contained in
the Prym locus in $\A_9$ and $\A_{12}$ and and we describe their
geometry.

In section 6 we describe the examples found in the ramified case. One
of the non-abelian examples gives a Shimura curve contained in the
ramified Prym locus in $\A_8$ and we show that it is not in the
Torelli locus.

In the appendix we describe the script and we give the table of the
examples.

\section*{Acknowledgements}

It is a pleasure to thank Jennifer Paulhus for sharing with us
the list of generating vectors for group actions on Riemann
surfaces. These data proved very helpful in double-checking our
computations.  

The authors thank IBM Power Systems Academic Initiative for
  providing a Linux server on which part of the GAP computations were
  performed.

\section{Special subvarieties of $\ag$}
\label{Shimura-section}

\begin{say}
  \label{VHS}
  Let $E : \Z^{2g} \times \Z^{2g} \ra \Z $ be the alternating form of
  type $(1,\lds, 1)$ corresponding to the matrix
  \begin{gather*}
    \begin{pmatrix}
      0 & I_g \\
      -I_g & 0
    \end{pmatrix}.
  \end{gather*}
  The Siegel upper half-space is defined as follows
  \begin{gather*}
    \sieg:= \{J \in \GL (\R^{2g}) : J^2 = - I, J^* E = E, E(x,Jx) >0,
    \ \forall x \neq 0 \}.
  \end{gather*}
  The group $\Sp(2g, \Z)$ acts on $\sieg$ by conjugation and this
  action is properly discontinuous. Set
  $\ag: = \Sp(2g, \Z) \backslash \sieg $.  This space has the both the
  structure of a complex analytic orbifold and the structure of a
  smooth algebraic stack.  Throughout the paper we will work with
  $\ag$ with the orbifold structure.  Denote by $A_J$ the real torus
  $\La_ \R / \La$ provided with the complex structure $J \in \sieg$
  and the polarization $E$.  It is a principally polarized abelian
  variety. On $\sieg$ there is a natural variation of rational Hodge
  structure, with local system $\sieg \times \Q^{2g}$ and
  corresponding to the Hodge decomposition of $\Co^{2g}$ in $\pm i$
  eigenspaces for $J$.  This descends to a variation of Hodge
  structure on $\ag$ in the orbifold or stack sense.
\end{say}

\begin{say}
  We refer to \S 2.3 in \cite{moonen-oort} for the definition of Hodge
  loci for a variation of Hodge structure.  A \emph{special
    subvariety} $\zg \subseteq\ag$ is by definition a Hodge locus of
  the natural variation of Hodge structure on $\ag$ described above.
  Special subvarieties contain a dense set of CM points and they are
  totally geodesic \cite[\S 3.4(b)]{moonen-oort}. Conversely an
  algebraic totally geodesic subvariety that contains a CM point is a
  special subvariety \cite{mumford-Shimura} (see
  \cite[Thm. 4.3]{moonen-linearity-1} for a more general result).  The
  simplest special subvarieties are the \emph{special subvarieties of
    PEL type}, whose definition is as follows (see \cite[\S
  3.9]{moonen-oort} for more details). 
  Given $J\in \sieg$, set
  \begin{gather}
    \label{endq}
    \End_\Q (A_{J}) := \{f\in \End( \Q^{2g}): Jf=fJ\}.
  \end{gather}
  Fix a point $J_0 \in \sieg$ and set $D:= \End_\Q (A_{J_0})$.  The
  \emph{PEL type} special subvariety $\zg (D)$ is defined as the image
  in $\ag$ of the connected component of the set
  $\{J \in \sieg: D \subseteq\End_\Q(A_J)\}$ that contains $J_0$.  By
  definition $\zg(D)$ is irreducible.
\end{say}

If $G\subseteq\Sp(2g,\Z)$ is a finite subgroup, denote by $\sieg^G$
the set of points of $\sieg$ that are fixed by $G$.  Set
\begin{gather}
  D_G:=\{ f\in \End_\Q (\Q^{2g}) : Jf=fJ, \ \forall J \in \sieg^G\}.
  \label{def-DG}
\end{gather}
In the following statement we summarize what is needed in the rest of
the paper regarding special subvarieties. See \cite[\S 3]{fgp} for the
proofs.
\begin{TEO}
  \label{bert}
  The subset $\sieg^G$ is a connected complex submanifold of $\sieg$.
  The image of $\sieg^G$ in $\ag$ coincides with the PEL subvariety
  $\zg (D_G)$.  If $J \in \sieg^G $, then
  $ \dim \zg(D_G) = \dim (S^2 \R^{2g})^G$ where $\R^{2g}$ is endowed
  with the complex structure $J$.
\end{TEO}

\section{Special subvarieties in the unramified Prym locus}
\label{prymetale}

In this section we explain how to construct Shimura subvarieties
generically contained in the Prym locus, that is contained in
$\overline{ \pr (\RR_{g}) }$ and intersecting ${ \pr (\RR_{g}) }$.
Recall that one has
$\overline{j(\mathsf{M}_{g-1} )} \subset \overline{ \pr (\RR_{g}) }$.
In fact it is known already from the work of Wirtinger \cite{wi} (see
\cite{b} for a modern proof) that Jacobians appear as limits of
Pryms. The fiber of the extended Prym map over a generic Jacobian has
been studied in detail in \cite{ds} and
\cite{kanev-global-Torelli}. It is therefore natural to extend the
search for Shimura subvarieties contained in the Torelli locus to the
case of the Prym locus and to ask whether such Shimura subvarieties
exist in high dimension.

For any integer $r \geq 3$ let $\Ga_r$ denote the group with
presentation $\Ga_r=\sx \ga_1, \lds, \ga_r | \ga_1\cds \ga_r =1\xs$.
A \emph{datum} is a pair $(G, \theta)$ where $G$ is a finite group and
$\theta : \Ga_r \lra G$ is an epimorphism. We will only be concerned
with the case $r=4$. If a datum $(G,\theta)$ is fixed, we set
$\mm :=(m_1, \lds, m_r)$ where $m_i $ is the order of
$ (\theta(\ga_i))$. We sometimes stress the importance of the vector
$\mm$ denoting a datum by $\datum$. (In fact this is important in the
\verb|MAGMA| script, which starts out by computing the possible
vectors $\mm$ that satisfy the Riemann-Hurwitz formula. So in the
computation the vector $\mm$ really comes before $(G,\theta)$.)

Denote by $\T_{0,r}$ the Teichm\"uller space in genus $0$ and with
$r\geq 4$ marked points.  The definition of $\T_{0,r}$ is as follows.
Fix $r+1$ distinct points $p_0, \lds, p_r$ on $S^2$.  For simplicity
set $P = (p_1, \lds, p_r)$.  Consider triples of the form
$(C, x, [f])$ where $C$ is a curve of genus 0, $x = (x_1, \lds, x_r) $
is an $r$-tuple of distinct points in $C$ and $[f]$ is an isotopy
class of orientation preserving homeomorphisms
$f : (C, x) \ra (S^2 , P)$.  Two such triples $(C, x , [f])$ and
$(C', x', [f'])$ are equivalent if there is a biholomorphism
$\phi: C \ra C'$ such that $\phi(x_i) = x'_i$ for any $i$ and
$ [f] = [f'\circ \phi ]$. The Teichm\"uller space $\T_{0,r}$ is the
set of all equivalence classes, see e.g. \cite[Chap. 15]{acg2} for
more details.  Since $C$ has genus 0 we can assume that $C=\PP^1$.
Using the point $p_0 \in S^2 - P$ as base point we can fix an
isomorphism $\Ga_r \cong \pi_1(S^2 - P, p_0)$.

If a datum $(G, \theta)$ and a point $t=[\PP^1, x, [f]] \in \T_{0,r}$
are fixed, we get an epimorphism
$\pi_1 ( \PP^1 - x, f\meno(p_0) ) \cong \Ga_r \ra G$ and thus a
covering $C_t \ra \PP^1=C_t/G$ branched over $x$ with monodromy given
by this epimorphism. The curve $C_t$ is equipped with an isotopy class
of homeomorphisms to a fixed branched cover $\Sigma $ of $S^2$.  Thus
we have a map $\T_{0,r} \ra \T_g \cong \T(\Sigma)$ to the
Teichm\"uller space of $\Sigma$.  The group $G$ embeds in the mapping
class group of $\Sigma$, denoted $\Mod_g$.  This embedding depends on
$\theta$ and we denote by $\gtheta \subset \Mod_g$ its image.  It
turns out that the image of $\T_{0,r}$ in $\T_g$ is exactly the set of
fixed points $\T_g^\gtheta$ of the group $\gtheta$.  We denote this
set by $\T\datu$. It is a complex submanifold of $\T_g$.  The image of
$ \T\datu$ in the moduli space $\mathsf{M}_g$ is a $(r-3)$-dimensional
algebraic subvariety that we denote by $\mathsf{M} \datu$.  See e.g.
\cite{gavino,baffo-linceo,clp2} and \cite [Thm. 2.1]{broughton-equi}
for more details.

In the discussion above the choice of the base point $p_0$ is
irrelevant.  On the other hand the choice of the isomorphism
$\Ga_r \cong \pi_1(S^2 -P, p_0)$ does matter.  To describe this we
introduce the braid group:
\begin{gather*}
  {\bf B_r}:=\langle \tau_1, \ldots ,\tau_{r}| \, \tau_i \tau_j =
  \tau_j \tau_i \, \, \text{for} \, \, |i-j|\geq 2, \,
  \tau_{i+1}\tau_i\tau_{i+1}=\tau_i\tau_{i+1}\tau_i \rangle
\end{gather*}
There is a morphism $\phi : \braid \ra \Aut(\Ga_r)$ defined as
follows:
\begin{gather*}
  \phi( \tau_i) (\gamma_i) = \gamma_{i+1}, \quad \phi(\tau_i)
  (\gamma_{i+1}) = \gamma_{i+1} ^{-1} \gamma_i \gamma_{i+1}, \\
  \phi(\tau_i) (\gamma_j ) = \gamma_j \quad \text{for }j \neq i, i+1.
\end{gather*}
From this we get an action of $\braid$ on the set of data:
$ \tau \cd (\mm, G, \theta) : = (\tau (\mm), G, \theta \circ
\phi(\tau\meno))$,
where $\tau (\mm) $ is the permutation of $\mm$ induced by $\tau$.
Also the group $\Aut(G)$ acts on the set of data by
$\alfa \cd (\mm, G, \theta) : = (\mm, G, \alfa \circ \theta )$.  The
orbits of the $\braid \times \Aut(G)$--action are called \emph{Hurwitz
  equivalence classes} and elements in the same orbit are said to be
related by a \emph{Hurwitz move}. Data in the same orbit give rise to
distinct submanifolds of $\T_g$ which project to the same subvariety
of $\M_g$.  So the submanifold $\T\datu$ is not well-defined, but the
subvariety $\M\datu$ is well-defined.  For more details see
\cite{penegini2013surfaces,baffo-linceo,birman-braids}.

\begin{DEF}
  \label{etaledatumdef}
  A \emph{Prym datum} is triple $\x=\dat$, where $\TG$ is a finite
  group, $\ttheta : \Ga_r \ra \TG$ is an epimorphism and
  $\sigma \in Z(\TG) $ is an element of order 2, that does not lie in
  $\bigcup_{i=1}^r \sx \ttheta (\ga_i)\xs$. (Here $Z(\TG)$ denotes the
  centre of $\TG$.)
\end{DEF}
Set $G:= \TG / \sx \sigma \xs$ and denote by $\theta : \Ga_r \ra G$
the composition of $\ttheta$ with the projection $\TG \ra G$.  A Prym
datum gives rise to two submanifolds of Teichm\"uller spaces, namely
$\mathsf{T} (G, \theta) \subset \T_g$ and
$\mathsf{T} (\TG, \ttheta) \subset \T_\tg$. Both are isomorphic to
$\T_{0,r}$ as explained above.  For any $t \in \T_{0,r}$ we have a
diagram
\begin{equation*}
  \begin{tikzcd}[row sep=tiny]
    \tilde{C}_t \arrow{rr}{\pi_ t} \arrow{rd} & &  \arrow{ld } C_t   = \tilde{C}_t /\sx \sigma\xs.  \\
    & \PP^1 &
  \end{tikzcd}
\end{equation*}
Here $\tilde{C}_t \ra \PP^1$ is the $\tG$-covering corresponding to
$t\in \T_{0,r}$ and to the datum $\tdatu$. The quotient map
$\pi_t : \tC_t \ra \tC_t/\sx \sigma\xs$ is an \'etale double cover.
In fact the elements of $\TG$ that have fixed points belong to some
conjugate of some $\sx \ttheta(\ga_i)\xs$. Since $\sigma$ is central
the definition ensures that it acts freely on $\tilde{C}_t$. Finally
it is easy to check that $C_t \lra \PP^1$ is the $G$-covering
corresponding to $t \in \T_{0,r}$ and to the datum $\datu$.  Denote by
$\eta_t$ the element of $\Pic^0(C_t)$, corresponding to the covering
$\pi_t$, i.e. such that
$(\pi_{t})_*(\OO_{\tilde{C}_t} )= \OO_{C_t} \oplus \eta_t$.
Associating to $t\in \T_{0,r}$ the class of the pair $(C_t, \eta_t)$
we get a map $\T_{0,r} \lra \Rg$.  This map has discrete fibres.  We
denote by $\RR (\x)$ its image.  Hence $ \dim \RR(\x) = r-3$.  The
following diagram (where $\tj$ and $j$ denote the Torelli morphisms)
summarizes the construction.
\begin{equation}
  \label{diamma1}
  \begin{tikzcd}[row sep=small]
    \Tei_{0,r} \arrow{r}{\cong} \arrow{d}[swap]{\cong} & \Tei(\TG,\ttheta) \arrow {ld} {\cong}     \arrow{rd} \arrow{rr} &   & \M_{\tg}  \arrow {r} {\tj }   & \mathsf{A}_\tg  \\
    \Tei(G,\theta) \arrow{rd} & & \RR (\x) \arrow{r}{\pr} \arrow{ru}
    \arrow{ld} & \A_{g-1} &
    \\
    & \M_g \arrow{r}{j} &  \A_g &  & \\
    & & & &
  \end{tikzcd}
\end{equation}

Given a Prym datum $\x=\dat$ fix an element $\tilde{C}_t$ of the
family $\Tei(\TG, \ttheta)$ with corresponding \'etale covering
$\pi_t: \tilde{C}_t \lra C_t$. For simplicty we drop the index $t$.
Set
\begin{gather*}
  V: = H^0(\tilde{C}, K_{\tilde{C}}),
\end{gather*}
and let $V= V_+ \oplus V_-$ be the eigenspace decomposition for the
action of $\sigma$.  The factor $V_+$ is isomorphic as a
$G$-representation to $H^0(C, K_C)$, while $V_-$ is isomorphic to
$H^0(C,K_C\otimes \eta)$.  Set
\begin{gather*}
  W := H^0( \tC , 2K_\tC ) ,
\end{gather*}
and let $W=W_+\oplus W_-$ be the eigenspace decomposition for the
action of $\sigma$.  We have $W_+ \cong H^0(C,2K_C)$ and
$W_- \cong H^0(C, 2K_C\otimes \eta)$ as $G$-representations.  The
multiplication map
\begin{gather*}
  m : S^2 V \lra W
\end{gather*}
is $\tG$-equivariant and is the codifferential of the Torelli map
$\tj : \M_\tg \ra \A_\tg$ at $[\tC] \in \M_\tg$.  We have the
following isomorphisms
\begin{gather*}
  (S^2V)^\tG = (S^2V_+)^G \oplus (S^2 V_-)^G,\qquad W^\tG = W_+ ^G.
\end{gather*}
Therefore $m$ maps $(S^2V)^\tG $ to $W_+^G$.  We are interested in the
restriction of $m$ to $ (S^2 V_-)^G $ that for simplicity we denote by
the same symbol:
\begin{gather}
  \label{nostram}
  m : (S^2 V_{-,t})^G \lra W_{+,t}^G.
\end{gather}
By the above discussion this is just the multiplication map
\begin{gather*}
  (S^2 H^0 (C, K_C\otimes \eta))^G \lra H^0(C, 2K_C)^G.
\end{gather*}

\begin{TEO}
  \label{teo1}
  Let $\x=\dat$ be a Prym datum. If there is $t\in \T_{0,r}$ such that
  the map $m$ in \eqref{nostram} is an isomorphism, then the closure
  of $\pr ( \RR(\x) )$ in $\A_{g-1}$ is a special subvariety of
  dimension $r-3$.
\end{TEO}

\begin{proof}
  Over $\Tei_{0,r}$ we have the families $\tC_t$, $C_t$,
  $\pi_t: \tC_t \ra C_t$ and $(C_t, \eta_t)$ as in diagram
  \eqref{diamma1}.  The lattice $H_1(\tilde{C}_t, \Z)$ is independent
  of $t\in \T_{0,r}$.  Set $\La : = H_1(\tilde{C}_t, \Z) _- $.  Call
  $Q$ the intersection form on $H_1(\tilde{C}_t, \Z)$, i.e. the
  principal polarization on the Jacobian of $\tC$.  Also $Q$ is
  independent of $t$.  Set
  \begin{gather*}
    E:=(1/2)\cd Q \restr{\La}.
  \end{gather*}
  $E$ is an integral symplectic form on $\La$.  Let $\HH_{g-1} $ be
  the Siegel upper half-space that parametrizes complex structures on
  $\La \otimes \R = H_1(\tilde{C}_t, \R)_-$ that are compatible with
  $E$.  For $t\in \T_{0,r}$ we have
  $H^1(\tilde{C}_t, \Co) = V_t \oplus \overline{V_t}$ with
  $V_t = H^0(\tilde{C}_t, K_{\tilde{C}_t})$ and also
  $ H^1(\tilde{C}_t, \Co)_- = V_{-,t} \oplus \overline{V_{-,t}}$.
  Dualizing we get the decomposition
  \begin{gather*}
    H_1(\tilde{C}_t, \Co)_- = V^*_{-,t} \oplus \overline{V^*_{-,t}}.
  \end{gather*}
  This decomposition corresponds to a complex structure $J_t$ on
  $H_1(\tilde{C}_t, \R)_-$, that is compatible with $E$ and therefore
  represents a point of $\HH_{g-1}$, that we denote by $f(t)$.  We
  have thus defined a map $ f: \T_{0,r} \ra \HH_{g-1}$.  The point is
  that the following diagram commutes:
  \begin{equation*}
    \label{diamma2}
    \begin{tikzcd}[row sep=small]
      \Tei_{0,r} \arrow{r}{f} \arrow{d} & \HH_{g-1}
      \arrow{d}  \\
      \RR(\x) \subset \Rg \arrow{r}{\pr} & \A_{g-1}.
    \end{tikzcd}
  \end{equation*}
  To check this it is enough to recall that
  \begin{gather*}
    P(C_t, \eta_t) = { V_{-,t}^* }/ { \La},
  \end{gather*}
  (see e.g. \cite[p. 295ff]{ACGH} or
  \cite[p. 374ff]{lange-birkenhake}).  Since $\TG$ preserves $Q$, $G$
  preserves $E$, so $G$ maps into $\Sp(\La, E)$.  Denote by $G'$ the
  image of $G$ in $\Sp(\La, E)$.  The complex structure $J_t$ is
  $G$-invariant, i.e. $f(t)=J_t \in \HH_{g-1}^{G'}$. Hence by Theorem
  \ref{bert} $ P (C_t,\eta_t) $ lies in the PEL special subvariety
  $\zg(D_{G'})$.  Therefore $\pr (\RR (\x) ) \subset \zg(D_{G'})$.
  Since $f (\T_{0,r}) \subset \HH_{g-1}^{G'}$ we can consider $f$ as a
  map $f: \T_{0,r} \ra \HH_{g-1}^{G'}$.  Recall that
  \begin{gather*}
    \Omega^1_{f(t)} \HH^{G'}_{g-1} \cong
    (S^2 H^0 (C_t, K_{C_t}\otimes \eta_t))^G =  (S^2 V_{-,t})^G,\\
    \Omega^1 _t \Tei_{0,r} \cong \Omega^1 _{[C_t]} \Tei (G, \theta)
    \cong H^0(C_t, 2K_{C_t})^G = W_{t,+}^G.
  \end{gather*}
  The codifferential is simply the multiplication map (see
  \cite{beau77} Prop. 7.5)
  \begin{gather*}
    m = (df_t)^* : (S^2 V_{-,t})^G \lra W_{t,+}^G.
  \end{gather*}
  This follows from the fact that the codifferential of the Torelli
  map restricted to $\T_{0,r}$ is the full multiplication map
  $S^2 V \ra W$.  By assumption there is a point $t\in \T_{0,r}$ such
  that the map $m$ is an isomorphism at $t$.  This implies first of
  all that $ \dim (S^2 V_{-,t})^G = \dim W_{t,+}^G = r-3$.  Moreover
  $f$ is an immersion at point $t$, hence its image has dimension
  $r-3$.  As the vertical arrows in \eqref{diamma1} are discrete maps,
  both $\pr (\RR(\x))$ and $\zg(D_{G'}) $ have dimension $r-3$.  Since
  $\pr(\RR(\x) ) \subset \zg(D_{G'})$ and $ \zg(D_{G'})$ is
  irreducible we conclude that
  $\overline{\pr( \RR(\x)) } = \zg(D_{G'})$ as desired.
\end{proof}

The Shimura subvarieties constructed using Theorem \ref{teo1}
intersect the Prym locus and are contained in its closure.

We wish to apply Theorem \ref{teo1} to construct examples of
1-dimensional special subvarieties (i.e. \emph{Shimura curves}) in
$\A_{g-1}$.  So from now on we assume $r=4$.

In the case $r=4$ the sufficient condition in Theorem \ref{teo1}
(namely that $m$ be an isomorphism) can be split in two parts:
\begin{gather}
  \label{condA}
  \tag{A} \dim (S^2V_-)^\TG = 1.  \\
  \label{condB}
  \tag{B} m : (S^2V_-)^\TG \lra W_+^G \text { is not identically } 0.
\end{gather}
Once \eqref{condA} is known, a sufficient condition ensuring
\eqref{condB} is the following
\begin{gather}
  \label{condB1}
  \tag{B1} (S^2V_-)^\TG \text{ is generated by a decomposable tensor}.
\end{gather}
In fact if $(S^2V_-)^\TG$ is generated by $
s_1 \otimes s_2 $
with $s_i \in V_-$, then $m(s_1 \otimes s_2) = s_1 \cd s_2$ which
cannot vanish identically.

\begin{remark}
\label{condizioneB1}
  We claim if \eqref{condA} holds, then \eqref{condB1} is equivalent
  to the fact that $(S^2 V_-)^{\tG} = W_1 \otimes W_2$ with $W_i$
  1-dimensional representations.  In one direction this is obvious. In
  the opposite direction, assume that \eqref{condA} and \eqref{condB1}
  hold. Let $V_- = W_1 \oplus \cds \oplus W_k$ be a decomposition in
  irreducible representations.  Then
  \begin{gather*}
    (S^2V_-)^\tG = \bigoplus_{i=1}^ k (S^2 W_i)^\tG \oplus
    \bigoplus_{i < j} (W_i \otimes W_j)^\tG.
  \end{gather*}
  Since $ (S^2V_-)^\tG$ is 1-dimensional, there are two cases: either
  $ (S^2V_-)^\tG = (S^2 W_i)^\tG$ for some $i$ or
  $ (S^2V_-)^\tG = (W_i \otimes W_j)^\tG$ for some $i$ and some $j$.
  We treat the first case, the other being identical. Let
  $t \in (S^2V_-)^\tG = (S^2 W_i)^\tG$ be a generator. By Schur lemma
  this represents an isomorphism $t : W_i^* \ra W_i$. If $d=\dim W_i$,
  then $t$ has rank $d$. By \eqref{condB1} $t$ is decomposable hence
  $d=1$, therefore $ (S^2V_-)^\tG = W_i\otimes W_i$.
\end{remark}

\begin{remark}
\label{condabelian}
By Remark \ref{condizioneB1}, if $\tilde{G}$ is abelian and condition \eqref{condA}  holds, then  condition  \eqref{condB1}
is automatically satisfied, since all the irreducible representations of an abelian group are 1-dimensional.  
\end{remark}

Finally we have to check which of the families satisfying conditions
\eqref{condA} and \eqref{condB} are generically contained in the Prym
locus, that is they are also generically irreducible.

Let us now recall the criterion given in \cite[p. 344]{mumford}.
Given an \'etale double covering
$\tilde{C} \ra C = \tilde{C}/\langle \sigma \rangle$, the associated
Prym variety $P(\tilde{C}, C)$ is reducible if and only if the curve
$C$ is hyperelliptic and denoting by $h$ a lift of the hyperelliptic
involution to $\tC$, we have $g(\tC/\langle h \rangle) >0$ and
$g(\tC/\langle h\sigma \rangle) >0$. In this case $P(\tC,C)$ is the
product
$J(\tC/\langle h \rangle) \times J(\tC/\langle h\sigma \rangle)$ as
principally polarised abelian variety.

\begin{LEM}
  \label{redu}
  Fix a Prym datum $\x=\dat$ satisfying \eqref{condA} and
  \eqref{condB}. If the generic Prym of the family is reducible, there
  exists a Prym datum with group $\tilde{H}$ containing $\tG$ also
  satisfying \eqref{condA} and \eqref{condB}, with a subgroup
  $\Z/2 \times \Z/2 \cong \langle h, \sigma \rangle \subset \tilde{H}$
  such that $\tC/\langle h, \sigma \rangle \cong \proj^1$.
\end{LEM}
\proof If the generic Prym $P(\tC,C)$ is reducible, by the above
criterion there exists a lifting $h$ of the hyperelliptic involution
of $C$ such that $\langle h, \sigma \rangle\cong\Z/2 \times \Z/2$ and
$\tC/\langle h, \sigma \rangle \cong \proj^1$.

Set $\tilde{H}:= \langle \tG, h \rangle$. If $\tG= \tilde{H}$ we are
done. If $\tG \subsetneq \tilde{H}$, the fixed point loci $\T(\tG)$
and $\T(\tilde{H})$ of the actions on the Teichm\"uller space
$\T_{\tg}$ coincide.

Clearly
$(S^2(H^0(K_{\tC}))^{\tilde{H}} \subset
(S^2(H^0(K_{\tC}))^{\tilde{G}}$. The multiplication map
\begin{gather*}
 (S^2(H^0(K_{\tC}))^{\tilde{G}} \ra H^0(2K_{\tC})^{\tG} = H^0(2K_{\tC})^{\tilde{H}} 
\end{gather*}
is an isomorphism of one dimensional vector spaces which is
$\tilde{H}$ equivariant. Hence also the multiplication map
$ (S^2(H^0(K_{\tC}))^{\tilde{H}} \ra H^0(2K_{\tC})^{\tilde{H} }$ is an
isomorphism.  This shows that $\tilde{H}$ defines a new Prym datum
satisfying \eqref{condA} and \eqref{condB} yielding the same family as
the one given by $\x=\dat$.  \qed

\section{Special subvarieties in the ramified Prym locus}
\label{sec:ramified}

In this section we would like to repeat the construction of the
previous section in the case in which the double covering
$\pi_t: {\tC}_t \to C_t$ is ramified at two points. This is the only
other case in which the associated Prym variety is principally
polarised \cite{mumford,lange-birkenhake}.

Let $C$ be a curve, $\eta$ a line bundle on $C$ of degree 1 and $B$ a
reduced divisor in the linear system $|{\eta^2}|$, i.e. $B=p+q$ with
$p\neq q$.  From this data one gets a double cover
$\pi:\tC \rightarrow C$ ramified over $B$.  The \emph{Prym variety}
$P(\tC,C)$ of $\pi$ is defined as the kernel of the norm map, which in
this case is connected.  As in the unramified case, the polarization
of $J(\tC)$ restricts to the double of a principal polarization $E$ on
$P(\tC,C)$. We will always consider $P(\tC,C)$ with the principal
polarizaztion $E$. In the case at hand it has dimension $g$.

Let $\RR_{g,[2]}$ denote the scheme parametrizing triples
$[C, \eta, B]$ up to isomorphism; the \emph{Prym map} is the morphism
\[ {\pr}: \RR_{g,[2]} \ra \ag
\]
which associates to $[C, \eta, B]$ the Prym variety $P(\tC,C)$ of
$\pi$.

We recall that we have the following inclusions
$\overline{j(\mathsf{M}_g )} \subset \overline{ \pr (\RR_{g,[2]}) }
\subset \overline{ \pr (\RR_{g+1}) }$.
Roughly the inclusion
$\overline{ \pr (\RR_{g,[2]}) } \subset \overline{ \pr (\RR_{g+1}) }$
can be seen as follows: given a double covering of a smooth curve of
genus $g$ ramified at two points, we obtain an admissible Beauville
covering gluing the two branch points and the corresponding
ramification points (see \cite{fl} p.763).

The inclusion
$\overline{j(\mathsf{M}_g )} \subset \overline{ \pr (\RR_{g,[2]}) }$
can be seen as follows: take a smooth genus $g$ curve $C$. Consider
the 2-pointed 1-nodal curve $X= C \cup \proj^1 $ where $C$ and
$\proj^1$ meet transversally at a point $x$ and let $p,q$ the two
marked points in $\proj^1$. Consider the admissible ramified double
cover $\tilde{X}$ of $X$ costructed as follows. Take the double cover
$f: \proj^1 \to \proj^1$ ramified in $p,q$ and denote by
$\{p_1,p_2\} = f^{-1}(x) \subset \proj^1$. Take two copies $C_1$,
$C_2$ of $C$, and glue these curves with $\proj^1$ identifying the
points $x \in C_i $ with $p_i$. Clearly the Prym $P(\tilde{X}, X)$ is
the Jacobian of $C$.

Thus it is again natural to extend the search for Shimura varieties in
the Torelli locus to the ramified Prym locus and the question about
the existence of such Shimura subvarieties in high dimension.

\begin{DEF}
  \label{ramiprym}
  A \emph{ramified Prym datum} is triple $\y=\dat$, where $\TG$ is a
  finite group, $\ttheta : \Ga_r \ra \TG$ is an epimorphism and
  $\sigma \in Z(\TG) $ is an element of order 2, that satisfies one of
  the following two conditions:
  \begin{enumerate}
  \item there is one and only one index $i$ such that
    $\sigma \in \sx \ttheta(\ga_i)\xs$ and $m_i = |\tG|/2$;
  \item there are exactly two indices $i, j$ such that
    $\sigma \in \sx \ttheta(\ga_i)\xs$,
    $\sigma \in \sx \ttheta(\ga_j)\xs$ and $m_j=m_i = |\tG|$.
  \end{enumerate}
  ($Z(\TG)$ denotes the centre of $\TG$.)
\end{DEF}
We set $G:= \TG / \sx \sigma \xs$ and we denote by
$\theta : \Ga_r \ra G$ the composition of $\ttheta$ with the
projection $\TG \ra G$.  The ramified Prym datum gives rise to two
submanifolds of Teichm\"uller spaces, namely
$\mathsf{T} (G, \theta) \subset \T_g$ and
$\mathsf{T} (\TG, \ttheta) \subset \T_\tg$. Both are isomorphic to
$\T_{0,r}$.  For any $t \in \T_{0,r}$ we have a diagram
\begin{equation*}
  \begin{tikzcd}[row sep=tiny]
    \tilde{C}_t \arrow{rr}{\pi_ t} \arrow{rd}[swap]{\psi} & &  \arrow{ld } C_t   = \tilde{C}_t /\sx \sigma\xs.  \\
    & \PP^1 &
  \end{tikzcd}
\end{equation*}
Here $\tilde{C}_t \ra \PP^1$ is the $\tG$-covering corresponding to
$t\in \T_{0,r}$ and to the datum $\tdatu$, while $C_t \ra \PP^1$ is
the $G$-covering corresponding to $\datu$. The quotient map
$\pi_t : \tC_t \ra \tC_t/\sx \sigma\xs$ has exactly two ramification
points.  To check this let $\{t_1,\lds,t_4\}$ be the critical values
of $\psi$.  If $\y $ satisfies condition (1) in Definition
\ref{ramiprym}, the two critical points of $\pi_t$ belong to the fibre
$\psi^{-1}(t_i)$ and thus $m_i = |\tG|/2$.  If $\y$ satisfies
condition (2) one critical point of $\pi_t$ is in $\psi^{-1}(t_i)$ and
the other is in $\psi^{-1}(t_j)$ and thus $m_i=m_j = |\tG|$.  Note
that $\tilde{g} = 2g$.

Denote by $\eta_t$ the element of $\Pic^0(C_t)$, corresponding to the
covering $\pi_t$, so that
$(\pi_{t})_*(\OO_{\tilde{C}_t} )= \OO_{C_t} \oplus \eta\meno_t$.  Let
$B_t \in |\eta_t^2|$ be the branch divisor of $\pi_t$.  Associating to
$t\in \T_{0,r}$ the class of the triple $(C_t, \eta_t, B_t)$ we get a
map with discrete fibres $\T_{0,r} \lra \rrd$.  Its image, denoted
$\RRd (\y)$, is $ (r-3)$-dimensional.  The following diagram
summarizes the construction.
\begin{equation}
  \label{diamma3}
  \begin{tikzcd}[row sep=small]
    \Tei_{0,r} \arrow{r}{\cong} \arrow{d}[swap]{\cong} & \Tei(\TG,\ttheta) \arrow {ld} {\cong}     \arrow{rd} \arrow{rr} &   & \M_{\tg}  \arrow {r} {\tj }   & \mathsf{A}_\tg  \\
    \Tei(G,\theta) \arrow{rd} & & \RRd (\x) \arrow{r}{\pr} \arrow{ru}
    \arrow{ld} & \A_{g} &
    \\
    & \M_g \arrow{r}{j} &  \A_g &  & \\
    & & & &
  \end{tikzcd}
\end{equation}

Given a ramified Prym datum $\dat$ and a covering
$\pi: \tilde{C} \lra C$ of the family, we have the eigenspace
decomposition for $\sigma$ just as in unramified case:
$V: = H^0(\tilde{C}, K_{\tilde{C}})= V_+ \oplus V_-$.  This time
$V_+\cong H^0(C, K_C)$ and $V_-\cong H^0(C,K_C\otimes \eta)$ as
$G$-modules.  Similarly $ W: = H^0( \tC , 2K_\tC )= W_+\oplus W_-$,
$W_+ \cong H^0(2K_C \otimes \eta^2) = H^0(2K_C +B)$ and
$W_- \cong H^0(C, 2K_C\otimes \eta)$.  The multiplication map
$ m : S^2 V \lra W $ is the codifferential of the Torelli map
$\tj : \M_\tg \ra \A_\tg$ at $[\tC] \in \M_\tg$.  It is
$\tG$-equivariant.  We have the following isomorphisms
\begin{gather*}
  (S^2V)^\tG = (S^2V_+)^G \oplus (S^2 V_-)^G,\\
  W^\tG = W_+ ^G.
\end{gather*}
Therefore $m$ maps $(S^2V)^\tG $ to $W_+^G$.  We are interested in the
restriction of $m$ to $ (S^2 V_-)^G $ that for simplicity we denote by
the same symbol:
\begin{gather}
  \label{nostramram}
  m : (S^2 V_-)^G \lra W_+^G.
\end{gather}
By the above discussion this is just the multiplication map
\begin{gather*}
  (S^2 H^0 (C, K_C\otimes \eta))^G \lra H^0(C, 2K_C \otimes \eta^2)^G
  \cong H^0(2K_C)^G \cong H^0(2K_{\tC})^{\tG}.
\end{gather*}

\begin{TEO}
  \label{teo1ram}
  Let $\y=\dat$ be a ramified Prym datum. If for some $t\in \T_{0,r}$
  the map $m$ in \eqref{nostramram} is an isomorphism, then the
  closure of $\pr ( \RR_{[2]} \y )$ in $\A_{g}$ is a special
  subvariety.
\end{TEO}
\begin{proof}
  Over $\Tei_{0,r}$ we have the families $\tC_t$, $C_t$, $ \eta_t$,
  $B_t$.  The lattice $H_1(\tilde{C}_t, \Z)$ the intersection form $Q$
  on $H_1(\tilde{C}_t, \Z)$ and the the sublattice
  $\La : = H_1(\tilde{C}_t, \Z) _- $ are independent of $t$.  Moreover
  $ E:=(1/2)\cd Q \restr{\La}$ is an integer-valued form on $\La$.
  Let $\HH_{g} $ be the Siegel upper half-space parametrizing complex
  structures on $\La \otimes \R = H_1(\tilde{C}_t, \R)_-$ that are
  compatible with $E$.  For any $t\in \T_{0,r}$ we have a
  decomposition
  $ H^1(\tilde{C}_t, \Co)_- = V_{-,t} \oplus \overline{V_{-,t}}$.
  Dualizing we get a decomposition
  $ H_1(\tilde{C}_t, \Co)_- = V^*_{-,t} \oplus \overline{V^*_{-,t}} $
  that corresponds to a complex structure $J_t$ on
  $H_1(\tilde{C}_t, \R)_-$. $J_t$ is compatible with $E$ and therefore
  represents a point of $\HH_{g}$, that we denote by $f(t)$.  We have
  thus defined a map $ f: \T_{0,r} \ra \HH_{g}$ that fits in following
  diagram:
  \begin{equation}
    \label{diamma2}
    \begin{tikzcd}[row sep=small]
      \Tei_{0,r} \arrow{r}{f} \arrow{d} & \HH_{g}
      \arrow{d}  \\
      \RRd(\y) \subset \rrd \arrow{r}{\pr} & \A_{g}.
    \end{tikzcd}
  \end{equation}
  The diagram commutes since also in this case \begin{gather*}
    P(\tC_t, C_t) = { V_{-,t}^* }/ { \La},
  \end{gather*}
  (see e.g. \cite[p. 295ff]{ACGH} or
  \cite[p. 374ff]{lange-birkenhake}). Since $\TG$ preserves $Q$, $G$
  preserves $E$, so $G$ maps into $\Sp(\La, E)$.  Denote by $G'$ the
  image of $G$ in $\Sp(\La, E)$.  The complex structure $J_t$ is
  $G$-invariant, i.e.  $f(t)=J_t \in \HH_{g}^{G'}$. Hence by Theorem
  \ref{bert} $ P (\tC_t,C_t) $ lies in the PEL special subvariety
  $\zg(D_{G'})$.  Therefore $\pr (\RR (\x) ) \subset \zg(D_{G'})$.
  Since $f (\T_{0,r}) \subset \HH_{g}^{G'}$ we can consider $f$ as a
  map $f: \T_{0,r} \ra \HH_{g}^{G'}$.  Recall that
  \begin{gather*}
    \Omega^1_{f(t)} \HH^{G'}_{g} \cong
    (S^2 H^0 (C_t, K_{C_t}\otimes \eta_t))^G  =  (S^2 V_{-,t})^G,\\
    \Omega^1 _t \Tei_{0,r} = \Omega^1 _{[C_t]} \Tei (G, \theta) =
    H^0(C_t, 2K_{C_t})^G \cong H^0(C_t, 2K_{C_t} \otimes \eta_t^2)^G =
    W_{t,+}^G.
  \end{gather*}
  The codifferential is simply the multiplication map
  \begin{gather*}
    m = (df_t)^* : (S^2 V_{-,t})^G \lra W_{t,+}^G
  \end{gather*}
  (see \cite{nagarama} Prop. 3.1, or \cite{lange-ortega}).  By
  assumption there is some $t\in \T_{0,r}$ such that the map $m$ is an
  isomorphism at $t$.  This implies first of all that
  $ \dim (S^2 V_{-,t})^G = \dim W_{t,+}^G = r-3$.  Moreover $f$ is an
  immersion at $t$, hence its image has dimension $r-3$.  As the
  vertical arrows in \eqref{diamma2} are discrete maps, both
  $\pr (\RR(\x))$ and $\zg(D_{G'}) $ have dimension $r-3$.  Since
  $\pr(\RR(\x) )\subset \zg(D_{G'})$ and $ \zg(D_{G'})$ is irreducible
  we conclude that $\overline{ \pr(\RR(\x)) } = \zg(D_{G'})$ as
  desired.
\end{proof}

The Shimura subvarieties constructed using Theorem \ref{teo1ram}
intersect the ramified Prym locus and are contained in its closure.

We wish to use Theorem \ref{teo1ram} to construct special curves. So
we set $r=4$.  Just as in the unramified case we can then split the
hypothesis of the Theorem in two conditions:
\begin{gather}
  \label{condA2}
  \tag{A} \dim (S^2V_-)^\TG = 1.
  \\
  \label{condB2}
  \tag{B} m : (S^2V_-)^\TG \lra W_+^G \text { is not identically } 0.
\end{gather}
Again once \eqref{condA2} is true, a sufficient condition ensuring
\eqref{condB2} is the following
\begin{gather}
  \label{condB3}
  \tag{B1} (S^2V_-)^\TG \text{ is generated by a decomposable tensor}.
\end{gather}

\section{Examples in the Prym locus}

In this section we discuss several examples of Shimura curves in the
Prym locus obtained using theorem \ref{teo1} and the scripts. Although
we do not study in detail all the examples gotten in this way (which
are listed in Tables 1 and 2) we give several informations for various
of them.  In particular for each example we recall the genera of $\tC$
and $C$, the group $\tG$ with a presentation and the monodromy,
i.e. the epimorphism $\ttheta$. With these data it is possible to
compute everything of the family, at least in principle, and such
presentation for all the examples of Tables 1 and 2 be found in the
lists on-line (see Appendix).

Before describing the examples, let us recall the description of two
Shimura families of Jacobians constructed in \cite{moonen-special} given by the equations:

$${\mathcal X}_t: v^2 = (u^3 +1)(u^3 + t), \quad  \quad \quad {\mathcal Y}_t: v^2 = u(u^2-1)(u^2 -t).$$
\

The first one is family (3) and the second one is family (4) in Table 1 in \cite{moonen-special}. These
two families will show up frequently in the following discussions. As observed in \cite{fgp} (see Table 1 and Table 2 in \cite{fgp}),
they have extra automorphisms: the group $D_6$ for (3)
and $D_4$ for (4), in fact (3)=(30) and (4)=(29) in the enumeration of
\cite{fgp}.  For every non-central element $a$ of order 2 in $ D_6$
and for any curve ${\mathcal X}_t$ in (3), the quotient ${\mathcal X}_t / \sx a \xs$ is
an elliptic curve $E_t$.  One easily shows that $J({\mathcal X}_t)$ is
isogenous to $E_t \times E_t$.

The same happens for (4) taking $E_t$ to be the quotient by a
non-central element of order 2 in $D_4$.  Therefore these two families
of Jacobians are both isogenous to the product of the same elliptic
curve $E \times E$ which moves.

We notice that many of the examples give rise to Pryms which are
isogenous to a product, but in dimension at least 4 they are all
irreducible. It would be interesting to study the decompotision up to
isogeny more in detail. For related questions in the case of Jacobians
see e.g. \cite {pauro}.

\begin{REM}
  \label{*}
  Notice that if one of the families of Pryms we constructed
  satisfying \eqref{condA} and \eqref{condB} is a family of Jacobians,
  it must satisfy condition $(*)$ of Theorem 3.9 in \cite{fgp}.  Hence
  if the dimension of the Pryms is $\leq 9$, they yield a Shimura
  curve that must appear in Table 2 of \cite{fgp}.
\end{REM}

\begin{LEM}
  \label{move}
  Let $\tdatu$ be a datum. Assume that for any $t\in \T_{0,r}$ there
  is a $\tG$-invariant rational Hodge substructure
  $W_t \subset H^1(P(\tC_t, C_t), \Co)$. If
  $(S^2 W^{1,0}_t)^\tG=\{0\}$, then the abelian variety corresponding
  up to isogeny to $W_t$ does not depend on $t$.
\end{LEM}

\begin{proof}
  It is enough to observe that the period matrix of the abelian
  variety corresponding up to isogeny to $W_t$ lies in
  ${\mathfrak {H}}_k^{\tG}$, where $k = \dim {W_t}^{1,0}$, and that
  $\mathfrak{H}_k^\tG$ is a point by the assumption.
\end{proof}

There are only 28 abelian examples satisfying condition \eqref{condA}, all in $\A_k$ with $k\leq 10$.
Recall that by Remark \ref{condabelian},  if the group is abelian and condition  \eqref{condA} holds, then \eqref{condB1} is also satisfied. 

Theorem \ref{teo1} tells us that these
families of Pryms yield special subvarieties of $\A_{k}$. We give here
a descriptions of the 7 examples with $k\geq 6$, for which the closure
of the Prym locus is not all $\A_k$. To verify that they are all
generically irreducible, we use Lemma \ref{redu} and the computer
check.

\subsection{The unramified abelian examples in $\A_6$ and in $\A_7$}

Note that for $k=6,7$, in the abelian examples we always have
$\TG = \Z/2 \times \Z/n$ and for these examples we give explicit
equations describing $\tC_t$ and $C_t$ as $n$-coverings of $\proj^1$,
via the quotient by $\Z/n$.

In the following $\zeta_n$ denotes a primitive $n$-th root of unity.

We denote by $\rho_n^i$ the character of $\sx h\xs=\Z/n$ mapping $h$
to $\zeta_n^i$, while $W_{\zeta_n^i }$ denotes the irreducible
representation of $\sx h\xs $ corresponding to this character,
i.e. mapping $h$ to $ \zeta_n^i$.  Since
$\sx h\xs \hookrightarrow \tG \ra G = \tG / \sx \sigma\xs$ is an
isomorphism, we consider $V_-$ as a representation of $\sx h\xs$.

\medskip\noindent \textbf{Example 29.}\\
$\tg = 13$, $g=7$,\\
$\tG =G(16,5)= \Z/2 \times \Z/8 = \langle g_2, g_1 \ | g_2^2=1, \
g_1^8 = 1, \
g_1g_2 =g_2 g_1\rangle $, $\sigma = g_1^4g_2$.\\
$ \ttheta(\ga_1) = g_1 , \quad \ttheta(\ga_2)= g_1^3, \quad \ttheta
(\ga_3) = g _1g_2, \quad \ttheta (\ga_4) = g _1^3g_2$.\\
$  \tilde{C}_t:\quad y^8=u^2(u^2-1)^7(u^2-t)^5\quad \pi:\tilde{C}\rightarrow \PP^1, \quad\pi(u,y)= u$.\\
$ g_2:(u,y)=(-u,-y),\quad  g_1(u,y)= (u,-\zeta_8y)=(u,\zeta_8^5y)\quad \sigma(u,y)=(-u,y)$.\\
$C_t: \quad
y^8=x(x-1)^7(x-t^2)^5\quad (x,y)=(u^2,y)$.\\
$V_-=W_{\zeta_8^2} \oplus 2W_{\zeta_8^5} \oplus W_{\zeta_8^6}\oplus
2W_{\zeta_8^7}\quad \quad (S^2V_-)^{\tG} \cong W_{\zeta_8^2} \otimes
W_{\zeta_8^6}$.

Here $P(\tC_t,C_t)$ is not isogenous to a Jacobian, since Table 2 of
\cite{fgp} does not contain families of genus 6 curves with an action
of $\Z/8$.  $P(\tC_t,C_t)$ is isogenous to the product of a fixed CM
abelian 4-fold $T'$ with a (Shimura) family of abelian surfaces with
an action of $\Z/4$.  Geometrically set
$D_1 := \tC/\langle g_1^4 \rangle$, $D_2 := \tC/\langle g_2 \rangle$,
$B:= \tC/\langle g_2, g_1^4 \rangle$.  Then
$g(D_2) = 7, g(D_1) = 5, g(B) = 3$,
$P(\tC,C) \sim P(D_2, B) \times P(D_1, B)$, where $T' = P(D_2, B)$,
while $P(D_1, B)$ is a Shimura family of abelian surfaces with an
action of $\Z/4$.

\medskip\noindent \textbf{Example 30.}\\
$\tg = 13$, $g=7$,\\
$\tG = G(20,5)= \Z/2 \times \Z/10 = \langle g_1, g_2 , g_3 \ | \ g_1^2
=1, g_2^2 = 1, g_3^5 =1\rangle$,
$\sigma = g_1g_2$,\\
$ \ttheta(\ga_1) = g_2 , \quad \ttheta(\ga_2)= g_2g_3, \quad \ttheta
(\ga_3) = g _1g_3^2, \quad \ttheta (\ga_4) = g _1g_3^2$.\\
$  \tilde{C}_t:\quad z^{10}=(u^2-1)(u^2-t),\quad \pi:\tilde{C}\rightarrow \PP^1, \quad\pi(u,y)= u$\\
$  g_2(u,z)=(-u,z),\quad  g_1(u,z)= (u,\zeta_{10}^2 z),\quad  g_3(u,z)=(u,\zeta_{10}^5z)\quad \sigma(u,z)= (-u,-z)$.\\
$C_t:\quad y^{10}=x^5(x-1)(x-t), \quad (x,y):=(u^2,u^{-1}z)$.\\
$V_-=W_{\zeta_{10}} \oplus W_{\zeta_{10}^2} \oplus
2W_{\zeta_{10}^4}\oplus W_{\zeta_{10}^7}\oplus W_{\zeta_{10}^8}\quad
\quad (S^2V_-)^{\tG} \cong W_{\zeta_{10}^2} \otimes W_{\zeta_{10}^8}$.

$P(\tC,C)$ is isogenous to $T \times A''$, where $T$ is a fixed CM
abelian surface and $A''$ is a moving abelian 4-fold.  Geometrically,
set $D_1:= \tC/\langle g_1 \rangle$, $D_2:= \tC/\langle g_2 \rangle$,
$F:= \tC/\langle g_1, g_2 \rangle$.  Then $g(D_1) = 6$, $g(D_2) = 4 $,
$ g(F)= 2$ and $T = P(D_2, F)$, $A'' = P(D_1, F)$.  Notice that $A''$
is not isogenous to any Shimura family of Jacobians, since Table 2 of
\cite{fgp} does not contain any family of Jacobians of genus 4 curves
admitting an action of $\Z/10$.

\medskip\noindent \textbf{Example 31.}\\
$\tg = 13$, $g=7$,\\
$\tG = G(24,9)= \Z/2 \times \Z/12= \langle g_1, g_2 , g_3 \ | \ g_1^4
=1, g_2^2 = 1, g_3^3 =1\rangle$, $\sigma = g_1^2g_2$.
\\
$ \ttheta(\ga_1) = g_2 , \quad \ttheta(\ga_2)= g_1, \quad \ttheta
(\ga_3) = g _3g_1^2, \quad \ttheta (\ga_4) = g _1g_2g_3^2$.\\
$  \tilde{C}_t:\quad z^{12}=(u^2-1)^3(u^2-t)^2,\quad \pi:\tilde{C}\rightarrow \PP^1, \quad\pi(u,y)= u$.\\
$  g_2(u,z)= (-u,z),\quad  g_1(u,z)= (u,\zeta_{12}^3 z),\quad  g_3(u,z)= (u,\zeta_{12}^4 z)$, $ \sigma(u,z)= (-u,-z)$.\\
$ C_t:\quad y^{12}=x^6(x-1)^3(x-t)^2, \quad (x,y):=(u^2,u^{-1}z)$.\\
$V_-=W_{\zeta_{12}^2} \oplus W_{\zeta_{12}^3} \oplus
W_{\zeta_{12}^4}\oplus W_{\zeta_{12}^5}\oplus
W_{\zeta_{12}^{10}}\oplus W_{\zeta_{12}^{11}}\quad \quad
(S^2V_-)^{\tG} \cong \zeta_{12}^{2} \otimes\zeta_{12}^{10}$.

Here $P(\tC,C)$ is isogenous to the product a fixed CM abelian 4-fold
$T''$ with the Shimura family (3) of \cite{moonen-special}. Set
$D_1:= \tC/\langle g_1^2 \rangle$, $D_2: = \tC/\langle g_2 \rangle$,
$F_1: = \tC/\langle g_1g_2 \rangle$,
$E_{\rho} = \tC/\langle g_1 \rangle$,
$E_{i} = \tC/\langle g_2, g_3 \rangle$ (these are the two CM elliptic
curves), $F:= \tC/\langle g_1^2, g_2 \rangle$. Then
$g(D_1) = g(D_2) = 4$, $g(F) =1$, $g(F_1) =2$,
$P(\tC, C) \sim P(D_1, F) \times P(D_2, F)$ and
$P(D_1, F) \sim J(F_1) \times E_{\rho}$ and $J(F_1)$ is the family (3)
of \cite{moonen-special}. Moreover, $P(D_2,F) \sim Y \times E_i$,
where $Y$ is a CM abelian surface, so
$T'' = Y \times E_{\rho} \times E_i$.

\medskip\noindent \textbf{Example 34.}\\
$\tg = 15$, $g=8$,\\
$\tG = G(24,9)= \Z/2 \times \Z/12 = \langle g_1, g_2 , g_3 \ | \ g_1^4
=1, g_2^2 = 1, g_3^3 =1\rangle$,
$\sigma = g_1^2g_2$. \\
$ \ttheta(\ga_1) = g_1^2, \quad \ttheta(\ga_2)= g_2g_3, \quad \ttheta
(\ga_3) = g _1g_3, \quad \ttheta (\ga_4) = g _1g_2g_3$.\\
$  \tilde{C}_t:\quad z^{12}=u^8(u^2-1)^6(u^2-t)^7,\quad \pi:\tilde{C}\rightarrow \PP^1, \quad \pi(u,y)= u$.\\
$  g_2(u,z)= (-u,z),\quad  g_1(u,z)= (u,\zeta_{12}^3 z),\quad  g_3(u,z)=(u,\zeta_{12}^4 z)\quad \sigma(u,z)= (-u,-z)$.\\
$C_t :\quad y^{12}=x^{10}(x-1)^6(x-t)^7, \quad (x,y):=(u^2,u^{-1}z)$.\\
$V_-=2W_{\zeta_{12}} \oplus W_{\zeta_{12}^2} \oplus
W_{\zeta_{12}^3}\oplus W_{\zeta_{12}^5}\oplus W_{\zeta_{12}^7}\oplus
W_{\zeta_{12}^8}\quad \quad (S^2V_-)^{\tG} \cong W_{\zeta_{12}^5}
\otimes W_{\zeta_{12}^7}$.

Here $P(\tC,C) \sim T''' \times E_{\rho} \times E_i \times E_{\rho}$,
where $T'''$ is a moving abelian fourfold not isogenous to a Jacobian,
since it carries an action of $\Z/2 \times \Z/12$ and in Table 2 of
\cite{fgp} there does not exist any family of Jacobians of genus 4
curves with an action of $\Z/12$.  More geometrically, set
$E := \tC/\langle g_1, g_2 \rangle$, $D_2:= \tC/\langle g_2 \rangle$,
$F:= \tC/\langle g_1^2, g_2 \rangle$,
$F_1:= \tC/\langle g_1g_2 \rangle$, $F_2:= \tC/\langle g_1 \rangle$,
$E_i \cong \tC/\langle g_2, g_3 \rangle$ (in fact it carries the
action of $\Z/4 \cong \langle g_1 \rangle$).  Then $g(D_1) = 4$,
$g(D_2) = 7$, $g(F ) = g(F_1) = g(F_2) = 2$,
$P(\tC, C) \sim P(F_1, E) \times P(F_2, E) \times P(D_2, F)$ and
$P(F_1, E) \sim P(F_2, E) \sim E_{\rho}$,
$P(D_2, F) \sim E_i \times T'''$.

\bigskip

\subsection{The unramified abelian examples in $\A_8$.}

We describe now the two only examples with $\tG$ abelian, yielding a
Shimura curve generically contained in the Prym locus in $\A_8$.  We
notice that up to now there are no known examples of Shimura varieties
generically contained in the Torelli locus for $g \geq 8$.  On the
other hand, by Remark \ref{*} these families are not families of
Jacobians since Table 2 in \cite{fgp} contains no example at all in
genus 8.

\medskip
\noindent \textbf{Example 35.}\\
$\tg= 17$, $g=9$.\\
$\tG = G(24,9)= \Z/2 \times \Z/12 = \langle g_1, g_2 , g_3 \ | \ g_1^4
=1, g_2^2 = 1, g_3^3 =1\rangle$,
$\sigma = g_1^2g_2$. \\
$ \ttheta(\ga_1) = g_1, \quad \ttheta(\ga_2)= g_1g_2, \quad \ttheta
(\ga_3) = g _1^3g_3, \quad \ttheta (\ga_4) = g _1g_2g_3^2$.\\
$V_-=W_{\zeta_{12}^3} \oplus W_{\zeta_{12}^9} \oplus
W_{\zeta_{12}^4}\oplus 2W_{\zeta_{12}^7}\oplus
2W_{\zeta_{12}^{11}}\oplus W_{\zeta_{12}^2}\quad \quad (S^2V_-)^{\tG}
\cong W_{\zeta_{12}^3} \otimes W_{\zeta_{12}^9}$.

We have $P(\tC,C) \sim P(D,E) \times A$, where $A$ is a fixed CM
abelian 5-fold and $D = \tC/\langle g_2,g_3 \rangle$,
$E = \tC/\langle g_2,g_3, \sigma \rangle$, $g(D) = 3$, $g(E) = 1$ and
$H^{1,0}(P(D,E)) \cong W_{\zeta_{12}^3} \oplus W_{\zeta_{12}^9}$.

\medskip
\noindent \textbf{Example 36.}\\
$\tg= 17$, $g=9$.\\
$\TG = G(32,21) = {\Z}/4 \times {\Z}/4 \times {\Z}/2 \cong \langle g_1
\rangle \times \langle g_2 \rangle \times \langle g_3 \rangle$,
\\
where $o(g_1)=o(g_2) =4$, $o(g_3) = 2$, $\sigma = g_2^2 g_3$.\\
$\ttheta(\gamma_1) = g_2, \ \ttheta(\gamma_2) = g_2g_3, \ \ttheta(\gamma_3) = g_1, \ \ttheta(\gamma_4) = g_1^3 g_2^2 g_3.$\\
$V_-=W_{0,1,0} \oplus W_{1,0,1} \oplus 2W_{1,1,0}\oplus W_{1,2,1}\oplus W_{1,3,0}\oplus W_{2,1,0}\oplus W_{3,1,0}$,\\
$(S^2V_-)^{\tG} \cong W_{1,3,0} \otimes W_{3,1,0}$,\\
where $W_{a_1,a_2,a_3}$ is the irreducible representation of the group
$\TG$ corresponding to the character $\rho_{a_1,a_2,a_3}$ mapping
$g_i$ to $\zeta_{k_i}^{a_i}$, for $i$ from 1 to $3$ ($k_1=k_2 =4$,
$k_3=2$).

Since $\TG$ is abelian both conditions \eqref{condA} and
\eqref{condB1} are satisfied. Theorem \ref{teo1} tells us that this
family of Pryms yields a special subvariety of $\A_{8}$.  Set
$E_1:= \tC/\langle g_1 \rangle$, $E_2:= \tC/\langle g_2 \rangle$,
$E_3:= \tC/\langle g_2g_3 \rangle$,
$E_4:= \tC/\langle g_1 g_2^2 g_3\rangle $. These are all elliptic
curves with a $\Z/4$-action, hence isomorphic to $E_i$.  We have
$H^0(E_1, K_{E_1}) \cong W_{0,1,0},$
$H^0(E_2, K_{E_2}) \cong W_{1,0,1},$
$H^0(E_3, K_{E_3}) \cong W_{1,2,1},$
$ H^0(E_4, K_{E_4}) \cong W_{2,1,0}$.  There is a diagram of coverings
\begin{equation}
  \label{diamma3bis}
  \begin{tikzcd}[row sep=tiny]
    & \tC \arrow {dd} \arrow {ld}
    \arrow{rd}  &    \\
    C_1= \tC/M\arrow{rd} {\pi_1}& & C_2= \tC/N \arrow{ld} {\pi_2}
    \\
    &F= \tC/H&
  \end{tikzcd}
\end{equation}
where $M = \langle g_3, g_1 g_2 \rangle$,
$N = \langle g_3, g_1 g_2^3 \rangle$ and
$H = \langle g_3, g_1 g_2, g_1 g_2^3 \rangle$. We have
$g(C_1) = g(C_2) = 3, \ g(F) = 1$, and
$H^{1,0}(P(C_1, F)) \cong W_{3,1,0} \oplus W_{1,3,0}, \
H^{1,0}(P(C_2,F)) \cong 2 W_{1,1,0}$. Hence
$$P(\tC,C) \sim E_1 \times E_2 \times E_3 \times E_4 \times P(C_1, F) \times P(C_2, F) = 4E_i \times P(C_1, F) \times P(C_2, F)$$
Since $(S^2(V_-))^{\TG} \cong S^2H^{1,0}( P(C_1, F) )$, by Lemma
\ref{move}, $P(\tC,C)$ is isogenous to the product of a fixed CM
abelian variety $A= 4E_i \times P(C_2, F)$ admitting an action of
$\Z/4$, with the Shimura family of abelian surfaces $P(C_1, F)$ having
an action of $\tilde{G}/M \cong \Z/4$ and moving in $\A_2(\Theta)$,
where $\A_2(\Theta)$ is the moduli space of abelian surfaces with a
given type of polarisation $\Theta$.

\subsection {The unramified abelian example in $\A_{10}$.} We now describe the only abelian unramified example in $\A_{10}$. \\

\noindent \textbf{Example 40.}\\
$\tg= 21$, $g=11$.\\
$\TG = G(32,3) = {\Z}/4 \times {\Z}/8 \cong \langle g_2 \rangle \times
\langle g_1 \rangle$,
where $o(g_1)=8$, $o(g_2) =4$,  $\sigma = g_2^2 g_1^4$.\\
$\ttheta(\gamma_1) = g_2, \ \ttheta(\gamma_2) = g_2g_1^4, \ \ttheta(\gamma_3) = g_1, \ \ttheta(\gamma_4) = g_1^3 g_2^2.$\\
$V_-=W_{0,1} \oplus 2W_{2,1} \oplus W_{2,3}\oplus W_{4,1}\oplus W_{5,0}\oplus W_{5,2}\oplus W_{6,1} \oplus W_{7,0}  \oplus W_{7,2}$,\\
$(S^2V_-)^{\tG} \cong W_{2,3} \otimes W_{6,1}$,\\
where $W_{a_1,a_2}$ is the irreducible representation of the group
$\TG$ corresponding to the character $\rho_{a_1,a_2}$ mapping $g_1$ to
$\zeta_{8}^{a_1}$, and $g_2$ to $\zeta_{4}^{a_2}$.

Since $\TG$ is abelian both conditions \eqref{condA} and
\eqref{condB1} are satisfied. Theorem \ref{teo1} tells us that this
family of Pryms yields a special subvariety of $\A_{10}$. Set
$F= \tC/\langle g_1 \rangle$, $ D= \tC/\langle g_2 \rangle$,
$Z= \tC/\langle g_2g_1^4 \rangle$, $X= \tC/\langle g_1^7 g_2 \rangle$,
$E= \tC/\langle g_1g_2, \sigma \rangle$,
$L = \tC/\langle g_1g_2^2 \rangle$. We have $g(F) = g(E) =g(L) = 1 $,
$g(D) = g(Z) = 2$, $g(X) = 3$,
$$P(\tC,C) \sim F \times L \times J(D) \times J(Z) \times P(X, E) \times P(Y, E),$$
where $H^0(F, K_F) = W_{0,1}$, $H^0(L, K_L) = W_{4,1}$,
$H^0(D, K_D) = W_{7,0} \oplus W_{5,0}$,
$H^0(Z, K_Z) = W_{5,2} \oplus W_{7,2}$, $H^{1,0}(P(X, E))= 2W_{2,1}$,
$H^{1,0}(P(Y, E))= W_{2,3} \oplus W_{6,1}$. Since
$(S^2(V_-))^{\TG} \cong S^2H^{1,0}( P(Y, E))$, by Lemma \ref{move},
$P(\tC,C)$ is isogenous to the product of a fixed CM abelian variety
$F \times L \times J(D) \times J(Z) \times P(X, E)$ with the Shimura
family of abelian surfaces $P(Y, E)$.

\subsection{Non abelian examples}

In this section we describe three non-abelian examples satisfying
condition \eqref{condA}, but not \eqref{condB1}.  We prove by ad hoc
arguments that condition \eqref{condB} holds.  Notice that these three
examples are examples of Shimura curves generically contained in the
Prym locus in $\A_g$, with $g = 9$ or $g= 12$.  Moreover by Remark
\ref{*}, Example 39 is not a family of Jacobians.

\medskip
\noindent\textbf{Example 39.}\\
$\tg= 19$, $g = 10$\\
$\TG = G(108,28) = (({\Z}/3 \times {\Z}/3) \rtimes {\Z}/3) \rtimes (
{\Z}/2 \times {\Z}/2) \cong$ \\
$\cong( (\langle g_4 \rangle \times \langle g_5 \rangle )\rtimes
\langle g_3 \rangle ) \rtimes (\langle g_1 \rangle
\times \langle g_2 \rangle)$,\\
where $o(g_4)=o(g_5) =o(g_3) = 3$, $o(g_1)=o(g_2) = 2$,\\
$Z(\TG) = \langle g_5, g_2 \rangle \cong {\Z}/3 \times {\Z}/2$,
\\
$g_3^{-1} g_4 g_3 = g_4 g_5$, $g_1^{-1} g_3 g_1 = g_3^{-1}$,
$g_1^{-1} g_4 g_1 = g_4^{-1}$, $\sigma = g_2$.\\
$\ttheta(\gamma_1) = g_1, \ \ttheta(\gamma_2) = g_1g_4, \
\ttheta(\gamma_3) = g_1g_2g_3, \ \ttheta(\gamma_4) = g_1 g_2
g_3g_4^2$.\\
Using \verb|MAGMA| we obtain the following decomposition in irreducible representations
$$V_-= V_{15} \oplus V_{16} \oplus V_{20} $$
(the notation is the one used by \verb|MAGMA|), 
$\dim(V_{15}) = \dim(V_{16}) = \dim(V_{20}) =3$. \\

The character table of $\tilde{G}$ and the formula
$$
\dim(V_i \otimes V_j)^{\tilde{G}} = \frac{1}{|\tilde{G}|}\sum_{h \in \tilde{G}} \chi_i(h) \chi_j(h),
$$ 
allow to check that 
$\dim (S^2(V_-))^{\TG} = \dim (V_{15} \otimes V_{20})^{\TG}=1$, hence
condition \eqref{condA} is satisfied.

We have to verify that also condition $(B)$ is satisfied. This is equivalent to showing that the family of Pryms moves, i.e. it is not isotrivial. 
This is implied by the following \\

{\bf Claim}. 
The  Prym variety $P(\tilde{C},C)$ is isogenous to a product of an abelian variety and the Jacobian $J(D)$  of a moving genus 2 curve $D$. 
\\

To prove the claim we check that $D= \tC/\langle g_1g_2, g_3 \rangle$
is a genus 2 curve such that
$H^0(D, K_D) \subset V_{15} \oplus V_{20} \subset V_-$. Finally, to
show that $J(D)$ moves we show that the curve $D$ moves as $\tC$
moves.

 Let $K := \langle g_1g_2, g_3 \rangle \cong S_3$. By Riemann-Hurwitz $\tilde{C}/K =:D$ has genus
2. 
The trace of $g_1$ on $V_{15}$ is $-1$. Since $g_1$ has order 2, we
have a decomposition 
$V_{15} = X_{15} \oplus W_{15}$, where $\dim(X_{15})=1$,
$\dim(W_{15})=2$ and ${g_1}_{|X_{15}} = Id_{X_{15}}$,
${g_1}_{|W_{15}} =- Id_{W_{15}}$. The same happens for
$V_{20} = X_{20} \oplus W_{20}$, where $\dim(X_{20})=1$,
$\dim(W_{20})=2$ and ${g_1}_{|X_{20}} = Id_{X_{20}}$,
${g_1}_{|W_{20}} =- Id_{W_{20}}$. The trace of $g_1$ on $V_{16}$ is 1,
so $V_{16} = X_{16} \oplus W_{16}$, where $\dim(X_{16})=2$,
$\dim(W_{16})=1$ and ${g_1}_{|X_{16}} = Id_{X_{16}}$,
${g_1}_{|W_{16}} =- Id_{W_{16}}$. Since $g_2$ acts as $-Id$ on
$V_- = V_{15} \oplus V_{16} \oplus V_{20}$ we have
${g_1g_2}_{|X_{j}} = -Id_{X_{j}}$, ${g_1g_2}_{|W_{j}} = Id_{W_{j}}$,
for $j=15,16,20$.

The group $S_3$ has three irreducible representations, $Y_1$, $Y_2$,
$Y_3$, where $\dim (Y_i)=1$, $i=1,2$, $\dim(Y_3) =2$, $Y_1$ is the
trivial one, $Y_2$ is the one given by the sign. 
If we look at the action of the subgroup
$K \cong S_3$ on $V_j$, $j = 15,16,20$, we see that we have
$V_{15} \cong Y_1 \oplus Y_3$ and the same for $V_{20}$, while
$V_{16} \cong Y_2 \oplus Y_3$. Hence the fixed point locus of the
action of $K$ on $V$, which we know to be two dimensional, since it is isomorphic to $H^0(D, K_D)$,  is given by
two copies of $Y_1$, one contained in $V_{15}$ and the other contained
in $V_{20}$. Therefore
$H^0(D, K_D) \subset V_{15} \oplus V_{20} \subset V_-$.  Hence
$P(\tC,C) \sim J(D) \times T$ for some 8-dimensional abelian
variety $T$.   To prove condition (B)
we will show that $J(D)$ moves.

Consider the action on $\tC$ of the subgroup
$L:= \langle g_1, g_2, g_3 \rangle \cong K \times \Z/2$.  By
Riemann-Hurwitz $\tC/L\cong \proj ^1$ and we have a factorisation
\begin{equation}
  \label{diamma4bis}
  \begin{tikzcd}
    \tC \arrow {rd} \arrow{r}{\phi} & D= \tC/K \arrow{d}
    [swap]{2:1}{p_D}
    \\
    & \proj^1= \tC/L= D/\langle g_2 \rangle.
  \end{tikzcd}
\end{equation}
If we prove that the 6 critical values of the hyperelliptic covering
$p_D$ move, we are done.  Denote by $\psi: \tC \ra \proj^1 = \tC/\TG$
the original covering and consider the factorisation
\begin{equation}
  \label{diamma5}
  \begin{tikzcd}
    \tC \arrow {d}{\psi}
    \arrow{r}{\phi}  & D= \tC/K \arrow{ld} {\pi} \\
    \proj^1= \tC/\TG. &
  \end{tikzcd}
\end{equation}
The $18:1$ covering $\pi$ factors as follows
\begin{equation}
  \label{diamma6}
  \begin{tikzcd}[row sep=tiny]
    D \arrow {dd}{\pi}
    \arrow{rd}{p_D}     & \\
    & D/\langle g_2 \rangle \cong \proj^1 \arrow{ld} {\pi'}
    \\
    \proj^1.&
  \end{tikzcd}
\end{equation}
Denote by $\{P_1,
P_2, P_3, P_4 \}$ the critical values of $\psi$ and by $\{y_1, y_2,
y_3, z_1, z_2, z_3\}$ the critical values of
$p_D$.
Looking at the above diagrams, one easily checks that the critical
values of $p_D$
all lie in $\pi'^{-1}(P_1) \cup \pi'^{-1}(P_2)$. More precisely:

$\pi'^{-1}(P_1)$
consists of 3 critical values $\{y_1,y_2,
y_3\}$ of $p_D$ which are regular for
$\pi'$
and of three critical points of order 2 for $\pi'$
which are regular values for $p_D$.

$\pi'^{-1}(P_2)$ consists of 3 critical values $\{z_1,z_2, z_3\}$ of
$p_D$ which are regular for $\pi'$ and of three critical points of
order 2 for $\pi'$ which are regular values for $p_D$.

$\pi'^{-1}(P_3)$ consists of three regular points and three critical
points of order 2 of $\pi'$ (all regular values for $p_D$).

$\pi'^{-1}(P_4)$ consists of two critical points of $\pi'$, one of
order 3 and one of order 6 (both regular values for $p_D$).

To understand better the $9:1$ map $\pi'$ let us consider this last
factorisation
\begin{equation}
  \label{diamma7}
  \begin{tikzcd}[row sep=tiny]
    \proj^1 = D/\langle g_2 \rangle \arrow {dd}{\pi'}
    \arrow{rd}{p_5}     & \\
    & \proj^1 = D/ \langle g_2, g_5 \rangle \arrow{ld} {\bar{\pi}}
    \\
    \proj^1= D/\langle g_2,g_4,g_5 \rangle.&
  \end{tikzcd}
\end{equation}
We have the following: $\bar{\pi}^*(P_i) = w_i + 2q_i$, for all
$i=1,2,3,4$, and $p_5^{-1}(w_1) = \{y_1,y_2,y_3\}$,
$p_5^{-1}(w_2) = \{z_1,z_2,z_3\}$.  The critical values of the Galois
$3:1$ covering $p_5$ are $w_4$ and $q_4$.

Consider the $3:1$ covering $\bar{\pi}: \proj^1 \ra \proj^1$.
Composing with automorphisms of $\proj^1$ in the source and in the
target, we can assume that $P_4 = \infty$, $P_3 = 0$, $P_2= 1$.  We
denote $P_1$ by the parameter $\lambda$, $w_4 =0$, $q_4 = \infty$,
$w_3 = 1$, and set $q_3=a$ for simplicity. Hence
$\bar{\pi}(z) = b \frac{(z-1)(z-a)^2}{z}$, where $b$ is nonzero.

Computing the derivative of $\bar{\pi}$ we see that the other two
critical points $q_1,q_2$ are $\frac{1\pm\sqrt{1+8a}}{4}$. Imposing
that $1$ and $\lambda$ are the corresponding critical values, we see
that $a, w_1,w_2$ are all non constant functions on $\lambda$.  We can
assume that $p_5(z) = z^3$, hence
$\{y_1, y_2, y_3 \} = p_5^{-1}(w_1) = \{ z \in \ \proj^1 \ | \ z^3 =
w_1\}$
and
$\{z_1, z_2, z_3 \} = p_5^{-1}(w_2) = \{ z \in \ \proj^1 \ | \ z^3 =
w_2\}$,
and since $w_1$ and $w_2$ are non-constant functions of $\lambda$, the
same holds for $y_i, z_i$, $ i=1,2,3$. This proves that as $\lambda$ varies, the hyperelliptic covering
$p_D:D \ra \proj^1$ varies, and hence the genus 2 curve $D$ varies, so
$J(D)$ varies. This proves the Claim.

A more detailed analysis shows that $P(\tC, C) \sim 3E \times 3 J(D)$, where $E= \tilde{C}/H$, where $H := \langle g_1, g_3 \rangle \cong S_3$. By Riemann Hurwitz one proves that $E$ has genus 1.

Moreover, looking at the action
 $H \cong S_3$ on $V_{j}$, $j =15,16,20$, one sees that  $H^0(E, K_E) \subset V_{16}$ and, since 
$(S^2(V_-))^{\TG} = (V_{15} \otimes V_{20})^{\TG}$ the elliptic
curve $E$ does not move by Lemma \ref{move}. 

\medskip
\noindent \textbf{Example 42.}\\
$\tg= 25$, $g = 13$.\\
$\TG = G(48,32) = {\Z}/2 \times SL(2, {\mathbb F}_3)\cong \langle g_1
\rangle \times SL(2, {\mathbb F}_3)$,
where\\
$SL(2, {\mathbb F}_3)= \bigg\langle g_2 =
\begin{pmatrix}
  2 & 1 \\ 2 & 0
\end{pmatrix}
\ g_3 =
\begin{pmatrix}
  0&2 \\ 1& 0
\end{pmatrix}
\ g_4 =
\begin{pmatrix}
  1& 2 \\ 2& 2
\end{pmatrix}
\ g_5 =
\begin{pmatrix}
  2 & 0 \\ 0 & 2
\end{pmatrix} \ | \ g_3^2 = g_4^2 = g_5,$ \\
\phantom{adadasdasdasdasdad}
$ \ g_5^2 = 1, \ g_2^3 =1, \ g_2^{-1} g_3 g_2 = g_4, \ g_2^{-1} g_4
g_2 = g_3 g_4, g_3^{-1} g_4 g_3
= g_4 g_5 \bigg\rangle$.\\
$\sigma = g_5,$
$Z(\tG) = \langle g_1, g_5 \rangle \cong \Z/2 \times \Z/2.$
\\
$\ttheta(\gamma_1) = g_1, \ \ttheta(\gamma_2) = g_1g_2g_5, \
\ttheta(\gamma_3) = g_1g_2 g_4 g_5, \ \ttheta(\gamma_4) = g_1 g_2
g_3g_4 g_5$.\\
$V_-= V_7 \oplus 2V_8 \oplus 2V_{11} \oplus V_{12}$.\\
Here $V_i$ are the irreducible representations as enumerated by
\verb|MAGMA|. Note that  $\dim(V_7) = \dim(V_8) = \dim(V_{11}) = \dim(V_{12})= 2$.\\
$(S^2(V_-))^{\TG} = (V_{8} \otimes V_{8})^{\tG} =
(\Lambda^2V_8)^{\tG}$
is one dimensional, hence condition \eqref{condA} is
satisfied. \\
We have to check condition \eqref{condB}.

Consider the commutative diagram:
\begin{equation}
  \label{diamma8}
  \begin{tikzcd} 
    & \tC \arrow {d}\arrow {ld}
    \arrow{rd}  &    \\
    C''= \tC/\langle g_1g_5 \rangle \arrow{dr} {p} & C'= \tC/\langle
    g_1 \rangle \arrow{d} {q} &
    C= \tC/\langle g_5 \rangle \arrow{dl}     \\
    &E= \tC/\langle g_1, g_5 \rangle.&
  \end{tikzcd}
\end{equation}
The curves $C'$ and $C''$ have genus 7, while $E$ has genus 1.  One
can check that $P(\tC, C) \sim P(C',E) \times P(C'', E)$, since
$H^{1,0}(P(C',E)) \cong V_7 + 2V_{11}$ and
$H^{1,0}(P(C'',E)) \cong 2V_8 + V_{12}$.  Since
$(S^2(V_-))^{\TG} = (V_{8} \otimes V_{8})^{\tG}$, the abelian variety
$P(C', E)$ does not move by Lemma \ref{move}.  To prove condition
\eqref{condB} we need to show that $P(C'', E)$ moves.

We have
$(S^2(H^0(C'', K_{C''})))^{\TG} \cong (S^2(2V_{8} + V_{12}))^{\tG} +
(S^2V_{3})^{\tG} = ( \Lambda^2V_8)^{\tG} + (S^2V_{3})^{\tG} = (
\Lambda^2V_8)^{\tG}$,
as one can check. Therefore $(S^2(H^0(C'', K_{C''})))^{\TG}$ has
dimension 1. So the family $C'' \ra C''/H = \tC/\tG$, where
$H = \tG/\langle g_1g_5 \rangle \cong SL(2, {\mathbb F}_3)$, satisfies
condition $(*)$ of \cite{fgp}, i.e. the codifferential of the Torelli
map, i.e.  the multiplication map
$(S^2(H^0(C'', K_{C''})))^{H} \ra H^0(C'', 2K_{C''})^{H}$, is an
isomorphism. So this is the Shimura family (40) of \cite{fgp}. Since
$J(C'') \sim P(C'', E) \times E$ and $J(C'')$ moves, while $E$ is
fixed, $P(C'', E)$ necessarily moves. Therefore $P(\tC,C)$ moves as
well and condition $(B)$ is satisfied.  Notice that on $E$ there is an
action of $\langle g_2 \rangle \cong \Z/3$, hence $E= E_{\rho}$.

\medskip\noindent \textbf{Example 43}\\
$\tg= 25$, $g =13$.
$\TG = G(48,30) = A_4 \rtimes \Z/4 = A_4 \rtimes \langle g_1 \rangle$,\\
where
$A_4 = \langle g_3 = (123), g_4 = (12)(34), g_5 = (13)(24)$,\\
$ g_3^3 =1, \ g_4^2 = 1, \ g_5^2 = 1, \ g_3^{-1} g_4 g_3 = g_5, \
g_3^{-1} g_5
g_3 = g_4 g_5 , \ g_4 g_5 = g_5 g_4$,\\
$g_1^{-1} g_3 g_1 = g_3^2, g_1^{-1} g_4 g_1 = g_5, g_1^{-1} g_5 g_1 =
g_4 \rangle $.\\
$\sigma = g_2 = g_1^2, $ $Z(\tG) = \langle g_2\rangle \cong \Z/2$.\\
$\ttheta(\gamma_1) = g_1g_5, \ \ttheta(\gamma_2) = g_1g_4, \
\ttheta(\gamma_3) = g_1g_3 g_4, \ \ttheta(\gamma_4) = g_1g_3g_4 g_5$.\\
$V_-= 2V_3 \oplus 2V_5 \oplus 2V_{10}$,
where $\dim(V_3) = 1$, $\dim(V_5) = 2,$ $\dim(V_{10})=3$.\\
(Notation of \verb|MAGMA| as above.)\\
$(S^2(V_-))^{\TG} = (V_{5} \otimes V_{5})^{\tG} = (\Lambda^2V_5)^{\tG}
= \Lambda^2V_5$.\\
So $(S^2(V_-))^{\TG} = (V_{5} \otimes V_{5})^{\tG} = \Lambda^2V_5$ is
1-dimensional, hence condition \eqref{condA} is satisfied.  We check
now condition \eqref{condB}.

Consider the normal subgroup
$H := \langle g_4, g_5 \rangle \cong \Z/2 \times \Z/2 \lhd \tG$.  Set
$C' = \tC/H$. One sees that $\tC \ra C'= \tC/H$ is a $4:1$ \'etale
covering and $C'$ has genus 7.  Moreover
$H^0(C', K_{C'}) = 2V_3 + 2V_5 + V_2$.

Set $H': = \langle \sigma \rangle \times A_4$. The quotient
$E:= \tC/H'$ is a genus one curve and $H^0(E, K_E) = V_2$.  So we have
the following commutative diagram:
\begin{equation*}
  \label{diamma9}
  \begin{tikzcd}[row sep=tiny]
    & \tC \arrow {ld}
    \arrow{rd}  &   & \\
    C'= \tC/H \arrow{dr} {p} & &
    C= \tC/\langle \sigma \rangle \arrow{dl}{q} &    \\
    &E= \tC/H'.& &
  \end{tikzcd}
\end{equation*}
Hence $P(\tC, C) \sim P(C', E) \times A$, where $A$ is a fixed abelian
6-fold.  Consider the groups $L = H'/H < K = \tG/H$,
$L\cong \langle \sigma \rangle \times (A_4/H) \cong \Z/6$. We have
\begin{equation}
  \label{diamma10}
  \begin{tikzcd}[row sep=tiny]
    & C' \arrow {ld} {\phi}\arrow{dd}
    \arrow{rd}  &   &  \\
    D= C'/\langle g_3 \rangle \arrow{dr} {\pi}& &
    E'= C'/\langle g_2 \rangle \arrow{dl}{\bar{q}} &    \\
    &E.& &
  \end{tikzcd}
\end{equation}
Notice that $D$ has genus 3 and $\phi$ is a $3:1$ \'etale
covering. Moreover,
$H^0(C', K_{C'})^{\langle g_3 \rangle} \cong H^0(K_D) \cong V_2 \oplus
2V_3$,
hence $H^{1,0}(P(C', D)) \cong 2V_5$ and
$P(\tC,C) \sim P(C',D) \times A'$, where $A' $ is a fixed abelian
variety of dimension 8.

We want to show that $P(C',D)$ moves and hence yields a Shimura curve.

Since the map $\phi$ is a 3:1 \'etale covering, it corresponds to a
3-torsion line bundle $\eta$ on $D$ and the pairs $[C',D]$ vary in a
curve ${\mathcal B}$ in the moduli space ${\mathcal R'}_{3,3}$
parametrising pairs $[C',D]$ where $C'\rightarrow D$ is a $3:1$
\'etale Galois covering of a genus three curve $D$. Denote by
${\mathcal P}: {\mathcal R'}_{3,3} \rightarrow \A_4(\Theta)$ the
corresponding Prym map. To conclude we need to show that the
differential of the restriction of ${\mathcal P}$ to ${\mathcal B}$ is
injective. Notice that the image of
$d{\mathcal P}_{[C',D]}: T_{[C',D]}{\mathcal R'}_{3,3} \rightarrow
T_{P(C',D)} \A_4(\Theta)$
is contained in the $\Z/3$ invariant part of
$T_{P(C',D)} \A_4(\Theta)$. Therefore
\begin{equation*}
  d{\mathcal P}^{\vee}_{[D, \eta]}: (T^*_{P(C',D)}  \A_4(\Theta))^{\Z/3} \cong S^2H^{1,0}(P(C',D))^{\Z/3} \rightarrow  T^*_{[D, \eta]}{\mathcal R'}_{3,3} \cong H^0(2K_D).
\end{equation*}
Observe that
$H^{1,0}(P(C',D)) \cong H^0(K_D(\eta)) \oplus H^0(K_D(\eta^2)) $,
hence
\begin{equation*}
  S^2H^{1,0}(P(C',D))^{\Z/3} \cong H^0(K_D(\eta)) \otimes
  H^0(K_D(\eta^2))
\end{equation*}
and the codifferential is identified with the multiplication map
$$m: H^0(K_D(\eta)) \otimes H^0(K_D(\eta^2))  \rightarrow H^0(2K_D).$$

First of all we prove that $m$ is injective. Observe that injectivity
follows from the base point free pencil trick if we show that
$|K_D(\eta^2)|$ is base point free. In fact in this case the kernel of
$m$ would be $H^0(\eta) = 0$.

Let us now prove that $|K_D(\eta^2)|$ is base point free.

So assume that $|K_D(\eta^2)|$ has a base point $p \in D$. Then
$h^0(K_D(\eta^2)(-p)) = h^1(\eta(p)) = 2$, hence $h^0(\eta(p)) = 1$,
therefore there exists a point $q \in D$ such that
$\eta = \OO_D(q-p)$. By the commutativity of diagram \eqref{diamma10},
we know that $\eta = \pi^*(\eta_E)$, where $\eta_E$ is the 3-torsion
line bundle on $E$ corresponding to the $3:1$ \'etale covering
$\bar{q}$. In particular $\eta$ is invariant by the covering
involution $\iota$ of $\pi$.  Hence we have
$p-q \equiv \iota(p)- \iota(q)$, equivalently
$p + \iota(q) \equiv \iota(p) +q$, which is impossible since $D$ is
not hyperelliptic. In fact the family $D \rightarrow {\mathbb P}^1$ is
the family $(4)$ of \cite{moonen-special}, which is not hyperelliptic.

Now denote by $\alpha$ the line bundle on $E$ yielding the $2:1$
covering $\pi$. We have $K_D= \pi^*(\alpha)$. Via the projection
formula, the map $m$ can be identified with the multiplication map
$$m_E: H^0(\alpha \otimes \eta_E) \otimes H^0(\alpha \otimes \eta_E^2) \rightarrow H^0(\alpha^2) \subset H^0(\alpha^2) \oplus H^0(\alpha) \cong H^0(2K_D).$$
Notice that $H^0(\alpha^2)$ can be identified with the cotangent space
to the bielliptic locus at the point $D$ and the cotangent space
$T^*_{[C',D]}{\mathcal B}$ is identified to a 1 dimensional subspace
of it via the forgetful map
${\mathcal R'}_{3,3} \rightarrow {\mathcal M}_3$. Since
$\dim( H^0(\alpha^2) ) = 4$ and $m$ is injective, $m_E$ is an
isomorphism, hence the differential of the restriction of the Prym map
to ${\mathcal B}$ at the point $[C',D]$ is injective.  Therefore the
family $P(\tC,C)$ moves.

\section{Examples in the ramified Prym locus}

In this section we briefly describe the examples of families of
  ramified Pryms satisfying conditions \eqref{condA} and
  \eqref{condB}, hence yielding Shimura curves  contained
  in the ramified Prym locus.

\medskip\noindent \textbf{Example 1.}\\
$ \tg = 4, g=2, \quad \tG = \Z/6=\Z/2 \times \Z/3 = \langle g_1, g_2 \
| g_1^2=1, \ g_2^3 = 1, \ g_1g_2 =g_2 g_1\rangle . $\\
$ \ttheta(\ga_1) = g_2 , \quad \ttheta(\ga_2)= g_2^2, \quad \ttheta
(\ga_3) = g _1 g_2 , \quad \ttheta (\ga_4) = g _1 g_2^2$.\\
$\tilde{C}_t:\quad y^3=u(u^2-1)^2(u^2-t),\quad
\pi:\tilde{C}_t\rightarrow
\PP^1, \quad\pi(u,y)= u$.\\
$\sigma=g_1:(u,y)\to (-u,-y),g_2:(u,y)\to (u,\zeta_3 y)$.\\
$C_t:\quad z^3=x^2(x-1)^2(x-t)\quad (x,z)=(u^2,yu)$.\\
Let $\zeta_3^i$ denote the character of $\sx g_2\xs$ mapping $g_2$ to
$\zeta_3^i$.  Let $W_{\zeta_3^i}$ be the irreducible representation of
$\sx g_2\xs $ corresponding to the character $\zeta_3^i$.\\
As a representation of $\sx g_2\xs$ we have:
$ V_-= W_{\zeta_3} \oplus W_{\zeta_3^2} , \quad (S^2V_-)^{\tG} \cong
W_{\zeta_3} \otimes W_{\zeta_3^2}$.\\
In the notation of Magma
$ V_4 = W_ {\zeta_3} , \quad V_6 = W_{\zeta_3^2} $.  The orbit of
$W_{\zeta_3}$ under the action of $Gal(\Q(\zeta_3), \Q)$ is clearly
$\{W_{\zeta_3}, W_{\zeta_3^2}\}$.  The Pryms $P(\tC,C)$ form a
1-dimensional family of abelian surfaces with a $\Z /3 $-action.  This
yields a Shimura curve, hence it is family (3) of
\cite{moonen-special}.

\medskip\noindent \textbf{Example 2.}\\
$     \tg = 4, \qquad g=2$, \\
$ \tG = D_6=G(12,4)= \langle g_1, g_2, g_3 \ | \ g_1^2=g_2^2 =
g_3^3=1, \ g_1^{-1} g_3 g_1 = g_3^{-1},
\ g_1g_2 =g_2 g_1, \ g_2 g_3 = g_3 g_2 \rangle $,\\
$ \sigma= g_2$.\\
$ \ttheta(\ga_1) = g_1 , \quad \ttheta(\ga_2)= g_1g_2, \quad \ttheta
(\ga_3) = g _3 , \quad \ttheta (\ga_4) = g _2 g_3^2$.\\
We observe that this is the same family as in Example 1, since the
family of the curves $C$ is family (3) of \cite{moonen-special}. In
fact family (3) is equal to family (28) of \cite{fgp}.

\medskip
\noindent
In the following two examples we have\\
$     \tg = 8, \qquad g=4$, \\
$ \tG = \Z/2 \times \Z/5 = \langle g_1, g_2 \ | g_1^2=1, \ g_2^5 = 1,
\ g_1g_2 =g_2 g_1\rangle $,
$\sigma = g_1$.  \\
In both cases the family of Pryms is a 1-dimensional family of abelian
4-folds with an action of $\Z /5 $, that yields a Shimura curve .

\medskip
\noindent\textbf{Example 3.}\\
$ \ttheta(\ga_1) = g_2 , \quad \ttheta(\ga_2)= g_2^2, \quad \ttheta
(\ga_3) = g _1 g_2 , \quad \ttheta (\ga_4) = g _1 g_2$.\\
$\tilde{C}_t:\quad y^{5}=u^2(u^2-1)^2(u^2-t),\quad \pi:\tilde{C}_t\rightarrow \PP^1, \quad\pi(u,y)= u$.\\
$g_1(u,y)= (-u,y),\quad g_2(u,y)= (u,\zeta_5 y)$.\\
$C_t:\quad y^5=x(x-1)^2(x-t)\quad (x,y)=(u^2,y)$.
\\
$  V_-= W_{\zeta_5} \oplus 2W_{\zeta_5^3} \oplus W_{\zeta_5^4} $.\\
$  (S^2V_-)^{\tG} \cong W_{\zeta_5} \otimes W_{\zeta_5^4}$.\\

\medskip
\noindent\textbf{Example 4.}\\
$ \ttheta(\ga_1) = g_2 , \quad \ttheta(\ga_2)= g_2^2, \quad \ttheta
(\ga_3) = g _1 g_2^3 , \quad \ttheta (\ga_4) = g _1 g_2^4$.\\
$\tilde{C}_t:\quad y^{5}=u(u^2-1)^2(u^2-t),\quad \pi:\tilde{C}\rightarrow \PP^1, \quad\pi(u,y)= u$.\\
$g_1:(u,y)\to (-u,-y),\quad g_2:(u,y)\to (u,\zeta_5 y)$.\\
$C_t:=\quad z^5=x^3(x-1)^2(x-t)\quad (x,z)=(u^2,yu)$.
\\
$V_-=2W_{\zeta_5} \oplus W_{\zeta_5^2} \oplus W_{\zeta_5^3}$.\\
$(S^2V_-)^{\tG} \cong W_{\zeta_5^2} \otimes W_{\zeta_5^3}$.

In the next example the group $\tG$ is not abelian and condition
\eqref{condB1} is not satisfied.  We show with a geometrical argument
that that condition \eqref{condB} holds and therefore we get a Shimura
curve in $\A_4$.

\medskip\noindent \textbf{Examples 5}
\\
$  \tg = 8,  g=4$, \\
$\tG = G(24,10) \cong \Z/3 \times D_4 = $\\
$= \langle g_1, g_2, g_3 \ | \ g_1^2=g_2^2 = g_3^3 =1,
\  (g_2g_1)^4 =1, \ g_3 g_i = g_i g_3 \ i=1,2, \rangle \cong $\\
$\cong \langle g_3 \rangle \times \langle  x=g_2g_1, y=g_1 \ | \ x^4 = y^2 =1, \  yx = x^{-1}y \rangle$\\
$   \sigma = (g_2g_1)^2$\\
$ \ttheta(\ga_1) = g_2 , \quad \ttheta(\ga_2)= g_1, \quad \ttheta
(\ga_3) = g _3, \quad \ttheta (\ga_4) = g _1 g_2 g_3^2$.\\
$V_-= V_{14} \oplus V_{15}$,
where $\dim(V_{14}) = \dim(V_{15}) = 2$ (notation of \verb|MAGMA|),\\
$(S^2(V_-))^{\TG} = (V_{14} \otimes V_{15})^{\tG} $, it is one
dimensional, hence condition \eqref{condA} is satisfied.  We need to
check condition \eqref{condB}. Consider
$\langle g_1\rangle \cong \Z/2$ and set $D = \tC/\langle g_1\rangle$.
The quotient $\tC \rightarrow D$ is a double cover ramified in 6
points, hence $g(D) = 3$.  We have the following commutative diagram:
\begin{equation}
  \label{diammaeven}
  \begin{tikzcd}[row sep=tiny]
    & \tC \arrow {ld}
    \arrow{rd}  &    \\
    D= \tC/\langle g_1 \rangle \arrow{dr} {p} & &
    C= \tC/\langle \sigma \rangle \arrow{dl}{q}     \\
    &E= \tC/\langle \sigma, g_1 \rangle .&
  \end{tikzcd}
\end{equation}
Here $q$ is a double cover ramified in 6 points and $E$ is an elliptic
curve with an action of $\langle g_3 \rangle \cong \Z/3$, hence it is
constant.  From the above diagram one sees that
$P(\tC, C) \sim P(D,E) \times A$, where $A$ is an abelian surface.  To
prove that $P(\tC, C)$ moves, we will show that $P(D,E)$ moves. Since
$E$ is fixed, it is equivalent to show that $J(D)$ moves in a one
dimensional family.  Denote by $\psi: \tC \to {\proj}^1 = \tC/\tG$ our
original covering, by $P_1,P_2, P_3, P_4$ the branch points of $\psi$
and by $\pi: E \to E/\langle g_2, g_3 \rangle \cong \tC/\tG$. The
branch points of the map $\pi$ (given by the $\Z/6$-action on $E$) are
$P_1, P_3, P_4$, hence, since $E$ does not move, the three branch
points of the original map $\psi$, $P_1,P_3,P_4$ do not move,
therefore $P_2$ must move.  The map $p$ has 4 branch points
$\{e_1, e_2, e_3, e_4\} \subset E$, where $\pi(e_i) = P_2$ for
$i =1,2,3$, while $\pi(e_4) = P_4$.  Since $P_2$ moves, the three
branch points $\{e_1, e_2, e_3\}$ move, hence the covering $p:D \to E$
moves and so do $D$ and $J(D)$. This concludes the argument.\\

\medskip \noindent
The following two examples both have\\
$  \tg = 12, \qquad g=6$, \\
$ \tG = \Z/2 \times \Z/7 = \langle g_1, g_2 \ | g_1^2=1, \ g_2^7 = 1,
\ g_1g_2 =g_2 g_1\rangle $,
$\sigma = g_1$.\\
In both cases the family $P(\tC,C)$ is a 1-dimensional family of
abelian 6-folds with an action of $\Z /7 $, that yields a Shimura
curve.

\medskip\noindent\textbf{Example 6.}
$ \ttheta(\ga_1) = g_2 , \quad \ttheta(\ga_2)= g_2^3, \quad \ttheta
(\ga_3) = g _1 g_2^4 , \quad \ttheta (\ga_4) = g _1 g_2^6$.\\
$\tilde{C}_t:\quad y^{7}=u(u^2-1)^3(u^2-t),\quad \pi:\tilde{C}\rightarrow \PP^1, \quad\pi(u,y)= u$,\\
$g_1:(u,y)\to (-u,-y), g_2:(u,y)\to (u,\zeta_{7} y)$,\\
$C_t: \quad z^7=x^4(x-1)^3(x-t)\quad (x,z)=(u^2,yu)$.\\
$V_-=2W_{\zeta_{7}} \oplus W_{\zeta_{7}^2} \oplus 2W_{\zeta_{7}^3}\oplus W_{\zeta_{7}^5}$.\\
$(S^2V_-)^{\tG} \cong W_{\zeta_{7}^2} \otimes W_{\zeta_{7}^5}$.

\medskip\noindent\textbf{Example 7.}
$ \ttheta(\ga_1) = g_2 , \quad \ttheta(\ga_2)= g_2^3, \quad \ttheta
(\ga_3) = g _1 g_2^5 , \quad \ttheta (\ga_4) = g _1 g_2^5$.\\
$\tilde{C}_t:\quad y^{7}=u^3(u^2-1)^3(u^2-t),\quad \pi:\tilde{C}\rightarrow \PP^1, \quad\pi(u,y)= u$.\\
$g_1:(u,y)\to (-u,-y), g_2:(u,y)\to (u,\zeta_7 y)$.\\
$C_t:\quad z^7=x^5(x-1)^3(x-t)\quad (x,z)=(u^2,yu)$.\\
$V_-=2W_{\zeta_{7}} \oplus W_{\zeta_{7}^3} \oplus W_{\zeta_{7}^4}\oplus 2W_{\zeta_{7}^5}$.\\
$(S^2V_-)^{\tG} \cong W_{\zeta_{7}^3} \otimes W_{\zeta_{7}^4}$.

\medskip\noindent\textbf{Example 8.}
$  \tg = 14, \qquad g=7$, \\
$ \tG = G(18,2) = \Z/2 \times \Z/9 = \langle g_1, g_2 \ | g_1^2=1, \ g_2^9 = 1,
\ g_1g_2 =g_2 g_1\rangle $,
$\sigma = g_1$.\\
$ \ttheta(\ga_1) = g_2^3 , \quad \ttheta(\ga_2)= g_2, \quad \ttheta
(\ga_3) = g _1 g_2^7 , \quad \ttheta (\ga_4) = g _1 g_2^7$.\\
$\tilde{C}_t:\quad y^9=u^7(u^2-1)^6(u^2-t)^5,\quad \pi:\tilde{C}\rightarrow \PP^1, \quad(u,y)\to u.$\\
$g_1:(u,y)\to (-u,-y), g_2:(u,y)\to (u,\zeta_9 y)$.\\
$C_t:\quad z^9=x^8(x-1)^6(x-t)^5\quad (x,z)=(u^2,yu).$\\
$V_-=W_{\zeta_{9}} \oplus 2W_{\zeta_{9}^2} \oplus 2W_{\zeta_{9}^4}\oplus W_{\zeta_{9}^6} \oplus W_{\zeta_{9}^8} $  .\\
$(S^2V_-)^{\tG} \cong W_{\zeta_{9}} \otimes W_{\zeta_{9}^8}$.

In the next example that satisfies condition \eqref{condA}, the group
$\tG$ is not abelian and condition \eqref{condB1} does not hold.
Hence we show again with a geometrical argument that also condition
\eqref{condB} holds and therefore it gives a Shimura curve 
contained in the ramified Prym locus in $\A_8$. Notice that by Remark
\ref{*} it is not contained in the Torelli locus.

\medskip\noindent\textbf{Example 9.}
$  \tg = 16, \qquad g=8$, \\
$\tG = G(40,10) \cong \Z/5 \times D_4 =$\\
$= \langle g_1, g_2, g_3 \ | \ g_1^2=g_2^2 = g_3^5=1
(g_2g_1)^4 = 1,  \  g_3 g_i = g_i g_3 \ i=1,2, \rangle \cong$\\
$ \cong \langle g_3 \rangle \times \langle  x=g_2g_1, y=g_1 \ | \ x^4 = y^2 =1, \  yx = x^{-1}y \rangle\\
\sigma = (g_2g_1)^2$.\\
$ \ttheta(\ga_1) = g_2 , \quad \ttheta(\ga_2)= g_1, \quad \ttheta
(\ga_3) = g _3, \quad \ttheta (\ga_4) = g _1 g_2 g_3^{-1}$.\\
$V_-= V_{22} \oplus V_{23} \oplus 2V_{24}$, where
$\dim(V_{22}) = \dim(V_{23}) = \dim(V_{24})=2$
(notation of \verb|MAGMA|).\\
$(S^2(V_-))^{\TG} = (V_{22} \otimes V_{23})^{\tG} $ and it is one
dimensional, hence condition \eqref{condA} is satisfied.  We check now
condition \eqref{condB}.  Consider
$\langle g_1\rangle \cong \Z/2 \subset \tG$ and denote by
$D = \tC/\langle g_1\rangle$. One sees that $\tC \rightarrow D$ is a
double cover ramified in 10 points, hence $g(D) = 6$.  We have the
following commutative diagram:
\begin{equation}
  \label{diammaeven}
  \begin{tikzcd}[row sep=tiny]
    & \tC \arrow {ld}
    \arrow{rd}  &    \\
    D= \tC/\langle g_1 \rangle \arrow{dr} {p} & &
    C= \tC/\langle \sigma \rangle \arrow{dl}{q}     \\
    &F= \tC/\langle \sigma, g_1 \rangle.&
  \end{tikzcd}
\end{equation}
where $q$ is a double cover ramified in 10 points and $F$ is a genus 2
curve with an action of $\langle g_3 \rangle \cong \Z/5$.  Therefore
$F$ is a CM curve for any value of the parameter. It follows that $F$
is constant.  From the above diagram one sees that
$P(\tC, C) \sim P(D,F) \times A$, where $A$ is an abelian surface.
Therefore, to prove that $P(\tC, C)$ moves, we will show that $P(D,F)$
moves. Since $F$ is fixed, this is equivalent to show that $J(D)$
moves in a one dimensional family.  Denote by
$\psi: \tC \to {\proj}^1 = \tC/\tG$ our original covering, by
$P_1,P_2, P_3, P_4$ the branch points of $\psi$ and by
$\pi: F \to F/\langle g_2, g_3 \rangle \cong \tC/\tG$. The branch
points of the map $\pi$ (given by the $\Z/10$-action on $F$) are
$P_1, P_3, P_4$, hence, since $F$ does not move, the three branch
points of the original map $\psi$, $P_1,P_3,P_4$ do not move,
therefore $P_2$ must move.  The map $p$ has 4 branch points
$\{e_1, e_2, e_3, e_4\} \subset F$, where $\pi(e_i) = P_2$ for
$i =1,2,3$, while $\pi(e_4) = P_4$.  Since $P_2$ moves, the three
branch points $\{e_1, e_2, e_3\}$ move, hence the covering $p:D \to F$
moves and so $D$ (and $J(D)$) moves. This concludes the argument.

\newpage
\begin{center} {\bf Appendix}
\end{center}

This appendix gives the relevant information on the script and
contains the tables of all the Prym data, which satisfy condition
\eqref{condA}.  Table 1 is for \'etale Prym data, while Table 2 is for
the ramified Prym data.

To perform the calculations done in this paper we wrote a \verb|GAP4|
\cite{GAP4} and a \verb|MAGMA| \cite{MA} script, both of them are
available at:

\smallskip

\begin{center}
  \verb|http://www.dima.unige.it/~penegini/publ.html|
\end{center}

\smallskip
\noindent
We now describe the \verb|GAP4| program \verb|PrymGenerators_v2.gap|.

The main routine is the function \verb|PossibleGoodPrym|.  One fixes a
range for the genus of the covering curve $\tilde{C}$ (we used
$4 \leq \tilde{g} \leq 30$), a range for the number of branch points
of the covering $\tilde{C} \rightarrow \PP^1$ (we considered only the
case of 4 branch points) and the type $x$ of Prym.  The latter means
the following:: $x=1$ for \'etale Prym datum, $x=2$ for ramified Prym
datum satisfying (2) of Definition \ref{ramiprym}, $x=3$ for ramified
Prym datum satisfying (1) of Definition \ref{ramiprym}.  Once all
these data are fixed the program performs the following calculations.

\medskip

\begin{enumerate}
\item First it calculates all possible signature types (Group order,
  $\mathbf{m}$) for the coverings $\tilde{C} \rightarrow \PP^1$.
\item After that, the program calculates for each signature type all
  the Prym data up to Hurwitz equivalence. These are: a group $\tG$ of
  a fixed order, all spherical systems of generators (SSG) for $\tG$
  (images of $\ttheta$) of the fixed type $\mathbf{m}$ up to Hurwitz
  moves, and an order $2$ central element in $\tG$.  Here the script
  calls some parts of the script given in \cite{matteo2011} (in
  particular the function \verb|NrOfComponents|). We refer to the
  appendix of \cite{matteo2011} for an explanation of the algorithm.
\end{enumerate}

\medskip

  While looking for the Prym data in the unramified case we can forget
  from the very beginning the cyclic groups thanks to the following
  lemma.
  \begin{LEM}
    If $\datu$ is an unramified Prym datum, then $\tG$ is not cyclic.
  \end{LEM}
  \begin{proof}
    Assume by contradiction that $\tG = \sx x \xs$ with $o(x) = 2n$
    and let $\{x^{n_i}=\ttheta(\ga_i)\}_{i=1}^k$ be a set of
    generators for $\tG$. There is only one element of order 2 in
    $\tG$, namely $\sigma :=x^n$. It follows that
    $\sigma \in \sx a\xs$ if and only if $o(a)$ is even. Since
    $\sigma \not \in \sx x^{n_i} \xs$, $o(x^{n_i} ) $ is odd for any
    $i$. On the other hand if $a =x^s$, then $o(a) = 2n / (2n, s)$.
    Write $n= 2^p q$ and $n_i = 2^{p_i} q_i$ with $q$ and $q_i$
    odd. Then
    $ o(x^{n_i} ) = 2 ^{p+1 - \min\{ p+1, p_i\} } \cd \tfrac{q}{(q,
      q_i)}.  $
    As this number is odd, we have $p_i \geq p+1$, so $n_i$ is even
    for any $i$.  Then clearly $[n_i]_{2n}$ cannot generate $\Z/(2n)$,
    contradiction.
  \end{proof}

We used the \verb|GAP4| program because the algorithm for finding
inequivalent pairs $(\tG$, SSG$)$ up to Hurwitz moves is efficient and
quite fast. One can find the output of this program at the web page
\begin{center}
   \verb|http://www.dima.unige.it/~penegini/publ.html|
\end{center}

\medskip

The remaining computations are performed using a \verb|MAGMA| program
\verb|PrymMagma_v6|, that we now describe.

\medskip

\begin{enumerate}
\item The function \verb|GoodExample| calculates the dimension
  $N_1:= \dim(S^2V)^\tG$ using the script \verb|PossGruppigFix_v2Hwr|
  written for the paper \cite{fgp} (we refer to \cite{fgp} for
  explanations). The input for this function are the data previously
  calculated by \verb|PrymGenerators_v2.gap|.
\item The function \verb|ProjSSG| constructs an $SSG$ for the group
  $G$ (for the covering $C \rightarrow C/G \cong \PP^1$) compatible
  with the given $SSG$ of $\tilde{G}$.
\item Afterwards we calculate the dimension $N_2= \dim(S^2V_+)^G$,
  again with the function \verb|GoodExample|.
\item The function \verb|GoodPrym(N1,N2)| checks condition
  \eqref{condA} in the form $N_1-N_2=1$. If the condition is satisfied
  the program will print \verb|GOOD EXAMPLE|. The resulting lists are
  Table 1 and 2 here.
\item Finally the function \verb|IsGoodGood| checks condition
  \eqref{condB1}.
\end{enumerate}

\medskip

All the results are available at
\begin{center}
   \verb|http://www.dima.unige.it/~penegini/publ.html|
\end{center}

A brief explanation of the tables.

The tables list all Prym data with $\tg \leq 28$ satisfying conditions
\eqref{condA} and \eqref{condB1} up to Hurwitz equivalence. It also
contains all the non-abelian examples satisfying \eqref{condA} (but
not \eqref{condB1}) for which we have verified condition
\eqref{condB}. For each datum we list a number that identifies the
datum, the genera of $\tC$ and $C$, the group $\tG$ and its
\verb|MAGMA| \verb|SmallGroupId|. The last two columns contain
information about conditions \eqref{condB1} and \eqref{condB}.  There
is a checkmark for \eqref{condB1} if and only if \eqref{condB1} is
satisfied.  If \eqref{condB1} is true, then \eqref{condB} follows.
When there is a checkmark for \eqref{condB}, this means that we proved
that \eqref{condB} holds.

In the tables some data are grouped together because they differ only
by $\ttheta$.

We do not give the full presentation of $\tG$, nor the morphism
$\ttheta$, since that would take too much space. The complete
information is of course available at the page above.

The data satisfying \eqref{condB} yield Shimura curves in the Prym
loci.

\begin{center}  
\begin{eqnarray*}
\begin{array}{|c|c|c|c|c|c|c|}
\hline
n & g(\tilde{C}) &  g(C)& \tG & \verb|SmallGroupId| & B1 & B\\
\hline
      1 - 3  &        5  &  3  &   \ZZ/2 \times \ZZ/4  &  G(8,2)  & \checkmark  & \checkmark \\ 
\hline
       4  &        5  &  3  &     \ZZ/2 \times \ZZ/6   &  G(12,5)  & \checkmark  & \checkmark  \\ 
  \hline 
       5  &        5  &  3  &    (\ZZ/2 \times \ZZ/4) \rtimes \ZZ/2  & G(16,3)  & \checkmark   & \checkmark \\ 
  \hline 
       6 - 8 &        5 &  3  &   \ZZ/2 \times \ZZ/2 \times \ZZ/4    &  G(16,10)  & \checkmark & \checkmark   \\ 
  \hline 
       9  &        5 &  3  &     \ZZ/2 \times A_4  &  G(24,13)  & \checkmark & \checkmark   \\ 
  \hline \hline 
       10 &        7 &  4  &     \ZZ/2 \times \ZZ/6  &  G(12,5)  & \checkmark   & \checkmark    \\ 
  \hline 
       11  &     7 &  4  &     \ZZ/4 \rtimes \ZZ/4   &  G(16,4)   & \checkmark  & \checkmark \\ 
         \hline 
       12  &      7 &  4  &    \ZZ/2 \times Q_8  &  G(16,12)  & \checkmark  & \checkmark   \\ 
 \hline \hline 
        13  &    9 &  5  &     \ZZ/4 \times \ZZ/4   &  G(16,2)   & \checkmark  & \checkmark \\ 
  \hline 
       14 &   9 &  5  &     \ZZ/4 \rtimes \ZZ/4   &  G(16,4)  & \checkmark   & \checkmark \\ 
   \hline   
       15 - 10  &     9 &  5  &     \ZZ/2 \times \ZZ/8   &  G(16,5)   & \checkmark & \checkmark  \\ 
   \hline 
       20  - 21 &     9 &  5  &     \ZZ/2 \times \ZZ/2 \times \ZZ/4   &  G(16,10)  & \checkmark  & \checkmark  \\ 
  \hline 
       23  - 24&     9 &  5  &     \ZZ/2 \times \ZZ/2 \times \ZZ/6    &  G(24,15)   & \checkmark  & \checkmark \\
    \hline 
       25  &    9 &  5  &    \ZZ/2 \times ((\ZZ/2 \times \ZZ/4) \rtimes \ZZ/2)   &  G(32,22)  & \checkmark  & \checkmark \\
       \hline 
       26  &     9 &  5  &    \ZZ/4 \times D_4  &  G(32,25)  & \checkmark   & \checkmark \\
          \hline                  \hline  
         27  &    11 &  5  &    \ZZ/2 \times \ZZ/8   &  G(16,5)  & \checkmark  & \checkmark \\
       \hline 
       28  &    11 &  6  &     \ZZ/2 \times \ZZ/12    &  G(24,9) & \checkmark  & \checkmark  \\
   \hline \hline 
       29  &    13&  7  &     \ZZ/2  \times \ZZ/8    &  G(16,5)  & \checkmark  & \checkmark \\
  \hline 
    30  &    13 &  7  &   \ZZ/2  \times \ZZ/10  &  G(20,5) & \checkmark  & \checkmark  \\
             \hline 
     31 &    13 &  7  &    \ZZ/2  \times \ZZ/12  &  G(24,9)  & \checkmark  & \checkmark \\
  \hline 
       32  &    13 &  7  &    (\ZZ/2  \times \ZZ/8) \rtimes \ZZ/2   &  G(32,9)  & \checkmark  & \checkmark \\
    \hline 
      33 &    13 &  7  &    (\ZZ/4  \times \ZZ/4) \rtimes \ZZ/2   &  G(32,24)  & \checkmark & \checkmark  \\
   \hline \hline 
       34 &    15 &  8  &   \ZZ/2  \times \ZZ/12  &  G(24,9)  & \checkmark  & \checkmark  \\
   \hline \hline 
      35  &    17 &  9  &    \ZZ/2  \times \ZZ/12  &  G(24,9)  & \checkmark  & \checkmark \\
       \hline 
                    \hline
 \end{array}
\end{eqnarray*}
\end{center}

\begin{center}  
       \begin{eqnarray*}
\begin{array}{|c|c|c|c|c|c|c|}
\hline
n & g(\tilde{C}) &  g(C)& \tG & \verb|SmallGroupId| & B1 & B\\
\hline
       36  &    17 &  9  &      \ZZ/2 \times \ZZ/4 \times \ZZ/4   &  G(32,21) & \checkmark  & \checkmark  \\
                \hline 
       37 &    17 &  9  &      (\ZZ/2 \times \ZZ/12) \rtimes \ZZ/2   &  G(48,14) & \checkmark   & \checkmark \\
         \hline 
         38  &   17 &  9  &(\ZZ/4 \times \ZZ/4 \times \ZZ/2) \rtimes  \ZZ/2&  G(64,71)  &    \checkmark   & \checkmark \\
          \hline \hline 
       39  &    19 &  10  &      (\ZZ/2 \times \ZZ/3 \times \ZZ/3) \rtimes S_3  &  G(108,28)  &   & \checkmark \\
       \hline \hline 
       40 &    21 & 11  &      \ZZ/4 \times \ZZ/8   &  G(32,3) & \checkmark  & \checkmark  \\
                \hline 
       41 &    21 &  11  &    \ZZ/4 \times D_8  &  G(64,118)  & \checkmark  & \checkmark \\ 
             \hline \hline 
       42  &    25 & 13  &  \ZZ/2 \times SL(2, 3)&  G(48,32)  &  &  \checkmark \\
                 \hline 
       43  &    25 & 13  &    A_4 \rtimes \ZZ/4 &  G(48,30)  &  &  \checkmark \\
           \hline 
        \hline
    \end{array}
\end{eqnarray*}
\end{center}
\begin{center}
Table 1
\end{center}

\begin{eqnarray*}
\begin{array}{|c|c|c|c|c|c|c|}
\hline
n & g(\tilde{C}) &  g(C)& \tG & \verb|SmallGroupId| &B1 &  B\\
   \hline 
1  &    4 & 2  & \ZZ/6 &  G(6,2)  & \checkmark & \checkmark  \\  
   \hline 
2  &    4 & 2  & D_6 &  G(12,4)  & \checkmark  & \checkmark  \\ 
         \hline \hline 
 3  - 4&  8 & 4  &  \ZZ/10&  G(10,2)  & \checkmark &  \checkmark\\
 \hline 
5  &    8 & 4  & \ZZ/3 \times D_4 &  G(24,10)  &  &  \checkmark   \\ 
         \hline \hline 
6  - 7&  12 & 6  &  \ZZ/14&  G(14,2)  & \checkmark &  \checkmark\\
               \hline \hline 
8  &  14 & 7  &  \ZZ/18&  G(18,2)  & \checkmark &  \checkmark\\
      \hline \hline 
9 &  16 & 8  &  \ZZ/5 \times D_4 &  G(40,10)  & &  \checkmark  \\
         \hline \hline 
    \end{array}
\end{eqnarray*}
\begin{center}
Table 2
\end{center}

\end{document}